\documentclass[10pt]{amsart}
\usepackage{amsmath,amssymb,mathrsfs,color}
\usepackage[hidelinks]{hyperref}
\DeclareFontFamily{U}{matha}{\hyphenchar\font45}
\DeclareFontShape{U}{matha}{m}{n}{
      <5> <6> <7> <8> <9> <10> gen * matha
      <10.95> matha10 <12> <14.4> <17.28> <20.74> <24.88> matha12
      }{}
\DeclareSymbolFont{matha}{U}{matha}{m}{n}
\DeclareFontSubstitution{U}{matha}{m}{n}

\DeclareFontFamily{U}{mathx}{\hyphenchar\font45}
\DeclareFontShape{U}{mathx}{m}{n}{
      <5> <6> <7> <8> <9> <10>
      <10.95> <12> <14.4> <17.28> <20.74> <24.88>
      mathx10
      }{}
\DeclareSymbolFont{mathx}{U}{mathx}{m}{n}
\DeclareFontSubstitution{U}{mathx}{m}{n}

\DeclareMathDelimiter{\vvvert}{0}{matha}{"7E}{mathx}{"17}

\newtheorem{lemma}{Lemma}[section]
\newtheorem{theorem}[lemma]{Theorem}
\newtheorem{remark}[lemma]{Remark}
\newtheorem{prop}[lemma]{Proposition}
\newtheorem{coro}[lemma]{Corollary}
\newtheorem{definition}[lemma]{Definition}
\newtheorem{example}[lemma]{Example}
\newtheorem{assumption}[lemma]{Assumption}
\allowdisplaybreaks
\numberwithin{equation}{section}

\begin{document}

\title[Ergodicity of inhomogeneous Markov processes]
{Ergodicity of inhomogeneous Markov processes under general criteria}

\author{Zhenxin Liu}
\address{Z. Liu : School of Mathematical Sciences,
Dalian University of Technology, Dalian 116024, P. R. China}
\email{zxliu@dlut.edu.cn}

\author{Di Lu}
\address{D. Lu (Corresponding author): School of Mathematical Sciences,
Dalian University of Technology, Dalian 116024, P. R. China}
\email{diluMath@hotmail.com; ludidl@mail.dlut.edu.cn}

\subjclass[2010]{Primary: 60B10, 60J25; Secondary: 47D07, 60J10.}
\keywords{Ergodicity; inhomogeneous Markov process; invariant measure family; Feller transition semigroup.}

\begin{abstract}

This paper is concerned with ergodic properties of inhomogeneous Markov processes. Since the transition probabilities depend on initial times, the existing methods to obtain invariant measures for homogeneous Markov processes are not applicable straightforwardly. We impose some appropriate conditions under which invariant measure families for inhomogeneous Markov processes can be studied. Specifically, the existence of invariant measure families is established by either a generalization of the classical Krylov-Bogolyubov method or a Lyapunov criterion. Moreover, the uniqueness and exponential ergodicity are demonstrated under a contraction assumption of the transition probabilities on a large set. Finally, three examples, including Markov chains, diffusion processes and storage processes, are analyzed to illustrate the practicality of our method.
\end{abstract}

\maketitle

\section{Introduction}

In probability theory and dynamical systems, ergodicity is an important property that describes the long-time behavior of systems.
This property guarantees that mean values in time are equal to mean values in the phase space for real-valued functions.
Ergodic properties for homogeneous Markov processes have been widely studied (see e.g. \cite{BCG,DFG,MT1,MT2,MT3,WYH,ZY2}), in which the transition probabilities and transition operators are not dependent on initial times.
They are extensively developed by viewing the processes on state space $\mathcal{X}$ as dynamical systems on $\mathcal{X}^{\mathbb{R}}$ according to Kolmogorov's extension theorem (see e.g. \cite{DZ1,HM,RBL}).

Ergodicity can be typically obtained by the existence and uniqueness of invariant measures.
In particular, the existence of invariant measures for homogeneous Markov processes can be proved by either the tightness assumption or the existence of a Lyapunov function for Feller transition semigroups (see e.g. \cite{HM}). The uniqueness holds under the irreducible condition for strongly Feller transition semigroups (see e.g. \cite{DZ1}). Furthermore, under the criteria for exponential ergodicity in \cite{MT3}, there exist positive constants $M_{1}<1$, $M_{2}<\infty$ and a measurable function $f$ on $\mathcal{X}$ such that
\begin{equation*}
\|P^{t}(x,\cdot)-\nu\|_{f} \leq M_{2}f(x)M_{1}^{t},\quad t\in\mathbb{R}^{+},\quad x\in\mathcal{X},
\end{equation*}
where the norm
$\|P^{t}(x,\cdot)-\nu\|_{f}:=\sup_{|g|\leq f}|\int_{\mathcal{X}}g(u)P^{t}(x,du)-\int_{\mathcal{X}}g(u)\nu(du)|$, $P^{t}(x,\cdot)$ are the transition probabilities for homogeneous Markov processes and $\nu$ is the unique invariant measure. These results can also be obtained under contraction assumptions (see e.g. \cite{HM1,HM2}).
Besides, the coupling method is a powerful tool to study ergodicity of invariant measures (see e.g. \cite{WD,TL}).

Different from the literature mentioned above, the goal of this paper is to deal with inhomogeneous Markov processes. It has been recognized that such processes driven by time-dependent external forces are of importance in applications. For example, queueing models are used to produce desired approximate results for the appointment system problem in operational research  (see \cite{BW} for more details). Early studies on queueing models focused more on homogeneous cases, but as pointed out by Worthington and Wall \cite{WW}, where inhomogeneous queueing models were studied, the parameters are dependent on time in many real queueing situations. For this or more potential situations, inhomogeneous cases are worthy of thorough analysis.

To the best of our knowledge, there exist a few works concerning ergodic properties of inhomogeneous Markov processes. Among the limited literature, for finite Markov chains, Hajnal \cite{JH1} studied the ergodic behaviour where transition matrices are regular and proved the weak ergodicity in \cite{JH2}, and Iosifescu \cite{MI} studied ergodicity and asymptotic behaviour. For general Markov chains, Bena\"{i}m, Bouguet and Cloez \cite{MFB} studied the asymptotic properties related to those of homogeneous Markov processes. Zeifman and Isaacson \cite{AZDI} studied ergodicity of continuous-time Markov chains but only for finite state spaces. Due to the effect of the time parameter, it is challenging to extend ergodic properties to inhomogeneous Markov processes under general conditions.

In the present paper, we consider inhomogeneous Markov process $\{X(t);t\geq s\}$ for any $s\in\mathbb{R}$ on a Polish space $\mathcal{X}$ with the transition probabilities $P(\tau,x,t,\cdot)$ for any $t\geq\tau\geq s$ and $x\in\mathcal{X}$. Denote the corresponding transition operators on the space of probability measures by $P^{\ast}_{\tau,t}$ given in \eqref{Def-transition}. Since $P(\tau,x,t,\cdot)$ depends on $\tau$, the classical definition of invariant measures is no longer applicable.
We present a definition of invariant measure families for inhomogeneous Markov processes as follows:
a family of probability measures $\{\mu_{s}\}_{s\in\mathbb{R}}$ on $\mathcal{X}$ is said to be an invariant measure family for $\{X(t);t\geq s\}$ if for each $s\in\mathbb{R}$ and $t\geq s$,
\begin{equation*}
P^{\ast}_{s,t}\mu_{s} = \mu_{t}.
\end{equation*}
It is clear that $\{\mu_{s}\}_{s\in\mathbb{R}}$ degenerates into the invariant measure for homogeneous Markov processes when $P^{\ast}_{s,t}$ are homogeneous.

Under general conditions, we show the existence, uniqueness and exponential ergodicity of $\{\mu_{s}\}_{s\in\mathbb{R}}$. More specifically, we generalize the classical Krylov-Bogolyubov method and extend the construction of Lyapunov functions to allow for the existence of time parameter. Then the existence of invariant measure family $\{\mu_{s}\}_{s\in\mathbb{R}}$ is derived based on this generalized method or the Lyapunov criterion for Feller transition semigroups, as shown in Theorems \ref{Th-1existence} and \ref{Th-Lyapunov} respectively. As for the uniqueness, it is established in Theorem \ref{thm-uniquecontinuous} under a contractive condition. Exponential ergodicity is obtained for inhomogeneous Markov chains in Theorem \ref{thmTVdiscrete}, and obtained for Markov processes in Theorem \ref{thmTV}, namely, there exist constants $0<\alpha<1$ and $M=M(x)>0$ such that
\begin{equation*}
\|P(\tau-t,x,\tau,\cdot)-\mu_{\tau}\|_{TV}
\leq
M\alpha^{t},
\end{equation*}
for any $\tau\in\mathbb{R}$, $t\geq 0$ and $\mu_{\tau}$-almost surely $x$, where $\|\cdot\|_{TV}$ is the total variation norm given in Section \ref{sec2}.

The proposed conditions for ergodicity can be easily verified in practice. For example, the contractive condition for diffusion processes on $\mathbb{R}^{n}$ is readily satisfied as their diffusion coefficients have positive lower bounds (see Theorem \ref{Th-nondegenerate} for more details). It is worth stressing out that by utilizing invariant measure families, the proposed conclusions have the potential to engender multiple extensions to asymptotic theory for inhomogeneous Markov processes. The latter could for instance concern the law of large numbers and the central limit theorem. Therefore, the overall aim of this paper is to illustrate some general conditions for obtaining ergodicity, and thus provide some useful tools for further researches on inhomogeneous Markov processes. Except for ergodicity, properties including reversibility and entropy production rate for inhomogeneous Markov processes are also worth investigating, referring to those studied by Ge, Jiang and Qian \cite{GJQ1,GJQ2} for time-periodic Markov chains with denumerable state spaces.

The remainder of this paper is organized as follows. In Section \ref{sec2}, we introduce the necessary notations and definitions of Markov processes and norms, and present properties of uniformly almost periodic functions. We provide general conditions for the existence, uniqueness and exponential ergodicity of invariant measure families in Section \ref{sec4}. In Section \ref{sec5}, we illustrate three applications of our theory, including Markov chains on a countable state space, nondegenerate diffusion processes on $\mathbb{R}^{n}$ and storage processes.

\section{Preliminaries}\label{sec2}
Let us introduce some basics of inhomogeneous Markov processes and recall definitions of supremum norm and total variation norm.

In the inhomogeneous setting, it is natural to consider processes with time sets $[s,+\infty)$ for $s\in\mathbb{R}$ (see e.g. \cite{CL2,CL,DD,DR}).
Let $\mathcal{X}$ be a Polish space with Borel $\sigma$-field $\mathcal B(\mathcal X)$. Denote by $\mathcal{P}(\mathcal{X})$ the space of probability measures on $\mathcal{X}$. 
Given $s\in\mathbb{R}$ and an inhomogeneous Markov process $\{X(t);t\geq s\}$ on $\mathcal{X}$ with the transition probabilities $P(\tau,x,t,\cdot)$ for any $t\geq\tau\geq s$, $x\in \mathcal{X}$, we define the transition operators $P^{\ast}_{\tau,t}$ of $\{X(t);t\geq s\}$ by
\begin{equation}\label{Def-transition}
P^{\ast}_{\tau,t}\nu(A) = \int_{\mathcal{X}}P(\tau,x,t,A)\nu(dx),
\end{equation}
for any probability measure $\nu$ in $\mathcal{P}(\mathcal{X})$ and $A\subset\mathcal{X}$.
Define the transition semigroups $P_{\tau,t}$ of $\{X(t);t\geq s\}$ by
\begin{equation}\label{Def-semigroup}
P_{\tau,t}f(x) = \int_{\mathcal{X}}f(y)P(\tau,x,t,dy),
\end{equation}
for any bounded and measurable function $f:\mathcal{X}\rightarrow\mathbb{R}$, $t\geq\tau$ and $x\in\mathcal{X}$.
Denote by $\nu_{s}$ the initial distribution of $\{X_{t};t\geq s\}$. Let $\mu_{t}:=P^{\ast}_{s,t}\nu_{s}$ for $t\geq s$. It holds that for any $t\geq s$,
\begin{equation}\label{eqt}
P^{\ast}_{s,t}\mu_{s} = \mu_{t}.
\end{equation}
We say the family of probability measures $\{\mu_{t}\}_{t\in\mathbb{R}}$ is an invariant measure family for $\{X_{t};t\geq s\}$ or for $P$ if \eqref{eqt} is satisfied for each $s\in\mathbb{R}$ and $t\geq s$.

For any $\phi:\mathcal{X}\rightarrow\mathbb{R}$, define the supremum norm of $\phi$ by $\|\phi\|_{\infty}=\sup_{x\in\mathcal{X}}|\phi(x)|$. Let $\mu$ be a finite signed measure on $\mathcal{X}$. Denote the total variation norm of $\mu$ (see (13.7) in \cite{MT}) by
\begin{equation*}
\|\mu\|_{TV}
:=
\sup_{\text{measurable\ } \phi, \|\phi\|_{\infty}\leq 1}\bigg|\int_{x\in\mathcal{X}}\phi(x)\mu(dx)\bigg|
 =
\sup_{A\subset\mathcal{X}}\mu(A)-\inf_{A\subset\mathcal{X}}\mu(A).
\end{equation*}
It holds that the difference of two probability measures is a finite signed measure.

In the following, we introduce the definition and properties of uniformly almost periodic functions which are used in Subsections \ref{sec4.2} and \ref{sec4.3}.
Let $(\mathcal{X}_{1},d_{1})$ and $(\mathcal{X}_{2},d_{2})$ be Polish spaces. Denote by $C(\mathbb{R}\times\mathcal{X}_{1},\mathcal{X}_{2})$ the space of all continuous functions $f:\mathbb{R}\times\mathcal{X}_{1}\rightarrow\mathcal{X}_{2}$.

\begin{definition} \rm\label{defalmostperiodic}(See \cite{YT1})
Let $f\in C(\mathbb{R}\times\mathcal{X}_{1},\mathcal{X}_{2})$. $f(t,x)$ is said to be almost periodic in $t$ uniformly for $x\in\mathcal{X}_{1}$, if for each $\epsilon>0$ and any compact set $D\subset\mathcal{X}_{1}$ there exists an $L>0$ such that any interval on $\mathbb{R}$ of length at least $L$ contains some point $s$ satisfying that for all $t\in\mathbb{R}$ and all $x\in D$,
\begin{equation*}
d_{2}(f(t+s,x),f(t,x))<\epsilon.
\end{equation*}
\end{definition}

For any two sequences $\{a_n^{'}\}$, $\{b_n^{'}\}\subset\mathbb R$, we say $\{a_n\}$, $\{b_n\}$ are common subsequences of $\{a_n^{'}\}$, $\{b_n^{'}\}$ if there is a sequence of indices $\{m_n\}$, such that for any $n$,
\begin{equation*}
   a_n=a_{m_n}^{'},\ b_n=b_{m_n}^{'}.
\end{equation*}
\begin{prop} \emph{(See \cite{YT1})}\label{prop-almostperiodic}
For a given continuous function $f\in C(\mathbb{R}\times\mathcal{X}_{1},\mathcal{X}_{2})$, the following statements are equivalent:
\begin{enumerate}
\item $f(t,x)$ is said to be almost periodic in $t$ uniformly for $x\in\mathcal{X}_{1}$;
\item for any sequence $\{\alpha'_{n}\}\subset\mathbb{R}$, there exists subsequence $\{\alpha_{n}\}\subset\{\alpha'_{n}\}$ such that $T_{\alpha}f:=\lim_{n\rightarrow\infty}f(t+\alpha_{n},x)$ exists uniformly on $\mathbb{R}\times K$ for any compact set $K\subset\mathcal{X}_{1}$;
\item for every pair of sequences $\{\alpha'_{n}\},\{\beta'_{n}\}\subset\mathbb{R}$, there exist common subsequences $\{\alpha_{n}\}$ and $\{\beta_{n}\}$ such that for any compact set $K\subset\mathcal{X}_{1}$,
    \begin{equation*}
    T_{\alpha+\beta}f = T_{\alpha}T_{\beta}f,
    \end{equation*}
    uniformly on $\mathbb{R}\times K$.
\end{enumerate}
\end{prop}

\section{Invariant measure families for inhomogeneous Markov processes}\label{sec4}

In this section we study the existence, uniqueness and exponential ergodicity of invariant measure families for inhomogeneous Markov processes.
Let $\mathcal{X}$ be a Polish space with Borel $\sigma$-field $\mathcal B(\mathcal X)$.
We consider inhomogeneous Markov process $\{X(t);t\geq s\}$ for any $s\in\mathbb{R}$ on $\mathcal{X}$ with the transition probabilities $P(\tau,x,t,\cdot)$ for $t\geq\tau\geq s$, $x\in \mathcal{X}$, the transition operators $P^{\ast}_{\tau,t}$ given in \eqref{Def-transition} and the transition semigroups $P_{\tau,t}$ given in \eqref{Def-semigroup}.

\subsection{Criteria for the existence of invariant measure families for $P$}

In this subsection, two criteria for the existence of invariant measure families are established.
Firstly, we generalize the Krylov-Bogoliubov theorem in \cite{DZ1} to obtain the existence of invariant measure families for $P$. Notice that it is always possible to restrict oneself to the case of discrete time for obtaining the existence of invariant measure families for $P$ by the following theorem.

\begin{definition} \rm
We say the transition semigroup $P_{n,n+m},n\in\mathbb{Z},m\in\mathbb{N}$ is Feller if $P_{n,n+m}\phi$ is continuous and bounded whenever $\phi$ is continuous and bounded.
\end{definition}

\begin{theorem}\label{Th-1existence}
Let $P_{n,n+m},n\in\mathbb{Z},m\in\mathbb{N}$ be Feller. If there exists $\mu\in\mathcal{P}(\mathcal{X})$ such that the set $\{\int_{\mathcal{X}}P(n-s,x,n,\cdot)\mu(dx)\}_{s\in\mathbb{N}^{+}}\subset\mathcal{P}(\mathcal{X})$ is tight for any $n\in\mathbb{Z}$, then there exists an invariant measure family for $P$.
\end{theorem}
\begin{proof}
Let $\{\mu_{n,s}\}$ be the sequence of probability measures defined by
\begin{equation*}
\mu_{n,s} = \frac{1}{s}\sum_{i=0}^{s-1}\int_{\mathcal{X}}P(n-i,x,n,\cdot)\mu(dx).
\end{equation*}
Since $\{\int_{\mathcal{X}}P(n-s,x,n,\cdot)\mu(dx)\}_{s\in\mathbb{N}^{+}}$ is tight, it holds that $\{\mu_{n,s}\}_{s\in\mathbb{N}^{+}}$ is also tight.
Then there exists a sequence $\{T_{N}\}_{N\in\mathbb{N}}$ such that the weak limit $\mu_{n}:=\lim_{N\rightarrow\infty}\mu_{n,T_{N}}$ exists for all $n\in\mathbb{Z}$ by the Prokhorov theorem and a diagonal argument. 

Let $\phi$ be any continuous and bounded function. For every $\epsilon>0$, there exists $K\in\mathbb{N}^{+}$ such that for any $N\geq K$,
\begin{align*}
&|\int_{\mathcal{X}}\phi(x)P_{n,n+1}^{\ast}\mu_{n,T_{N}}(dx)-\int_{\mathcal{X}}\phi(x)\mu_{n+1,T_{N}}(dx)|\\
=\:&
\bigg|\frac{1}{T_{N}}\int_{\mathcal{X}}\phi(x)\sum_{i=0}^{T_{N}-1}\int_{\mathcal{X}}(P(n-i,y,n+1,dx)-P(n+1-i,y,n+1,dx))\mu(dy)\bigg|\\
=\:&
\bigg|\frac{1}{T_{N}}\int_{\mathcal{X}}\phi(x)\int_{\mathcal{X}}(P(n-T_{N},y,n+1,dx)-P(n,y,n+1,dx))\mu(dy)\bigg|\\
\leq\:&
\frac{\epsilon}{3}.
\end{align*}
Reasoning as above, and using the facts that for any $n\in\mathbb{Z}$, $P_{n,n+1}$ is Feller and $\mu_{n}$ is the weak limit of $\mu_{n,T_{N}}$, we have
\begin{align*}
&\bigg|\int_{\mathcal{X}}\phi(x)P_{n,n+1}^{\ast}\mu_{n}(dx)-\int_{\mathcal{X}}\phi(x)\mu_{n+1}(dx)\bigg|\\
\leq \:& \frac{\epsilon}{3}+\bigg|\int_{\mathcal{X}}\phi(x)P_{n,n+1}^{\ast}\mu_{n}(dx)-\int_{\mathcal{X}}\phi(x)P_{n,n+1}^{\ast}\mu_{n,T_{N}}(dx)\bigg|+\bigg|\int_{\mathcal{X}}\phi(x)\mu_{n+1,T_{N}}(dx)\\
&
-\int_{\mathcal{X}}\phi(x)\mu_{n+1}(dx)\bigg|\\
\leq\:& 
\epsilon.
\end{align*}
Thus for any $m>n$, $m,n\in\mathbb{Z}$, using the fact that for any $1\leq i\leq m-n$, function $\int_{\mathcal X}\phi(x)P(n+i,\cdot,m,dx)$ is continuous and bounded, and similar arguments as above, we have
\begin{align*}
&\bigg|\int_{\mathcal{X}}\phi(x)P_{n,m}^{\ast}\mu_{n}(dx)-\int_{\mathcal{X}}\phi(x)\mu_{m}(dx)\bigg|\\
\leq\:&
\bigg|\int_{\mathcal{X}}\phi(x)P_{n,m}^{\ast}\mu_{n}(dx)-\int_{\mathcal{X}}\phi(x)P_{n+1,m}^{\ast}\mu_{n+1}(dx)\bigg|+\cdots+\bigg|\int_{\mathcal{X}}\phi(x)\mu_{m}(dx)\\
&
-\int_{\mathcal{X}}\phi(x)P_{m-1,m}^{\ast}\mu_{m-1}(dx)\bigg|\\
\leq\:&
(m-n)\epsilon.
\end{align*}
Since $\epsilon$ and $\phi$ are arbitrary, it holds that $P_{n,m}^{\ast}\mu_{n}=\mu_{m}$.

Notice that for any $t\in\mathbb{R}$, there exists $n\in\mathbb{Z}$ such that $t\geq n$. Define the sequence of probability measures $\{\nu_{t}\}_{t\in\mathbb{R}}$ by $\nu_{t}=\int_{\mathcal{X}}P(n,x,t,\cdot)\mu_{n}(dx)$. If we set $\tilde{\nu}_{t}=\int_{\mathcal{X}}P(m,x,t,\cdot)\mu_{m}(dx)$ for some integer $m$ satisfying $m\neq n$ and $m\leq t$, then for any continuous and bounded function $\phi$,
\begin{align*}
&\bigg|\int_{\mathcal{X}}\phi(x)\nu_{t}(dx)-\int_{\mathcal{X}}\phi(x)\tilde{\nu}_{t}(dx)\bigg|\\
=\:&\bigg|\lim_{N\rightarrow\infty}\frac{1}{T_{N}}\sum_{i=0}^{T_{N}-1}\int_{\mathcal{X}}\phi(x)(\int_{\mathcal{X}}P(n-i,z,t,dx)\mu(dz)-\int_{\mathcal{X}}P(m-i,z,t,dx)\mu(dz))\bigg|\\
\leq\:&\lim_{N\rightarrow\infty}\frac{2}{T_{N}}\|\phi\|_{\infty}\cdot|m-n|\\
=\:&0.
\end{align*}
It follows from the above argument that $\nu_{t}$ is well defined. On the other hand,
\begin{equation*}
P_{s,t}^{\ast}\nu_{s} = \int_{\mathcal{X}}\int_{\mathcal{X}}P(s,x,t,\cdot)P(n,y,s,dx)\mu_{n}(dy)=\nu_{t}.
\end{equation*}
Thus we have that $\{\nu_{t}\}_{t\in\mathbb{R}}$ is an invariant measure family for $P$.
\end{proof}

In the rest of this subsection, we provide another criterion for the existence of invariant measure families for $P$ using a Lyapunov function.

\begin{definition} \rm\label{Def-Lya}
A Borel measurable function $V:\mathbb{Z}\times\mathcal{X}\to[0,\infty]$ is called a Lyapunov function for the transition probabilities $P$ if $V$ satisfies the following conditions:
\begin{enumerate}
\item the set $B:=\{x:\sup_{m\in\mathbb{Z}}V(m,x)<\infty\}$ is nonempty;
\item for any $x\in\mathcal{X}$, $n\in\mathbb{Z}$ and $c\geq 0$, $V(n,x)\geq f(x)\geq 0$ for some function $f$ with the compact level set $\{x:f(x)\leq c\}$;
\item there exist $0<\gamma<1$ and $C>0$ such that for any $x\in B$ and $n\in\mathbb{Z}$,
      \begin{equation*}
      \int_{\mathcal{X}}V(n,y)P(n-1,x,n,dy) \leq \gamma V(n-1,x)+C.
      \end{equation*}
\end{enumerate}
\end{definition}

\begin{theorem}\label{Th-Lyapunov}
Let $P_{n,n+m},n\in\mathbb{Z},m\in\mathbb{N}$ be Feller. If the transition probabilities $P$ have a Lyapunov function $V$, then there exists an invariant measure family for $P$.
\end{theorem}
\begin{proof}
 For any $x\in B$, $n\in \mathbb{Z}$ and $m\in \mathbb{N}$, using $\gamma<1$ and $\sup_{m\in\mathbb{Z}}V(m,x)<\infty$, we have
\begin{align*}
\int_{\mathcal{X}}V(n,y)P(n-m,x,n,dy)
\leq\:&
\int_{\mathcal{X}}(\gamma V(n-1,z)+C) P(n-m,x,n-1,dz)\\
\leq
\:&\gamma^{m}V(n-m,x)+C\frac{1}{1-\gamma}\\
\leq\:&
M,
\end{align*}
for some $0<M=M(x)<\infty$ by iterating.
Therefore, using (ii) of Definition \ref{Def-Lya} and Chebyshev's inequality, we have for any $\epsilon>0$ and $c$ large enough,
\begin{align*}
P(n-m,x,n,\{x:f(x)\leq c\})\geq\:& P(n-m,x,n,\{x:V(n,x)\leq c\})\\
\geq\:& 
1-\frac{\int_{\mathcal{X}}V(n,y)P(n-m,x,n,dy)}{c}\\
>\:&
1-\epsilon,
\end{align*}
and consequently $\{P(n-m,x,n,\cdot)\}_{m\in\mathbb{N}^{+}}$ is tight, which contains a convergent subsequence.
Thus there exists an invariant measure family for $P$ by Theorem \ref{Th-1existence} in which $\mu=1_{x}$.
\end{proof}

\begin{remark}\rm
Lyapunov functions with the time parameter can be used to study stability properties of solutions of nonautonomous differential equations (see \cite{YT} for more details).
\end{remark}

\subsection{Criteria for the uniqueness and exponential ergodicity}

In this subsection, we show the uniqueness and exponential ergodicity of the invariant measure family for $P$.
We say the invariant measure family $\{\mu_{t}\}_{t\in\mathbb{R}}$ for $P$ is exponential ergodic with the total variation norm $\|\cdot\|_{TV}$
if there exist constants $0<\alpha<1$ and $M=M(x)>0$ such that
\begin{equation*}
\|P(\tau-t,x,\tau,\cdot)-\mu_{\tau}\|_{TV}
\leq
M\alpha^{t},
\end{equation*}
for any $\tau\in\mathbb{R}$, $t\geq 0$ and $\mu_{\tau}$-almost surely $x$.

For homogeneous Markov processes, Harris \cite{TEH} proved that the transition probabilities converge to the unique invariant measure at an exponential rate. Moreover, Hairer and Mattingly \cite{HM2} provided a skillful proof for the exponential rate. Then Hairer \cite{HM1} executed a slight simplification of the proof in \cite{HM2}. In this paper, we extend the idea of Hairer and Mattingly \cite{HM2} to inhomogeneous Markov processes.
The following two assumptions are used in obtaining the uniqueness and exponential ergodicity of invariant measure families for Markov chains.

\begin{assumption}\label{Assumption-exponential1}\rm
There exists a Borel measurable function $V:\mathcal{X}\to[0,\infty]$ such that
\begin{enumerate}
\item the set $\{x:V(x)<\infty\}$ is nonempty;
\item there exist $0<\gamma<1$ and $C>0$ such that for any $x\in \mathcal{X}$ and $n\in\mathbb{Z}$,
      \begin{equation*}
      \int_{\mathcal{X}}V(y)P(n-1,x,n,dy) \leq \gamma V(x)+C.
      \end{equation*}
\end{enumerate}
\end{assumption}

\begin{remark}\rm
The function $V$ in Assumption \ref{Assumption-exponential1} is different from the one in Definition \ref{Def-Lya}, since the domains are different and also there is no assumption about compactness for the function $V$ in Assumption \ref{Assumption-exponential1}. But in some ways, we can construct a function $V$ which satisfies Definition \ref{Def-Lya} from the one satisfying Assumption \ref{Assumption-exponential1}.
For example, we assume there exists a Borel measurable function $f:\mathcal X\to[0,\infty]$ with the compact level set $\{x;f(x)\leq c\}$ for any $c\geq 0$ such that Assumption \ref{Assumption-exponential1} is satisfied. Then the function 
\begin{equation*}
V:(n,x)\mapsto(\pi-\arctan n)f(x), \ (n,x)\in\mathbb Z\times\mathcal X,
\end{equation*}
satisfies items (i)-(iii) in Definition \ref{Def-Lya}.
\end{remark}

For any $R>0$, define the set $\mathcal{C}_{R}=\{(x,y):V(x)+V(y)\leq R\}$.

\begin{assumption}\label{Assumption-exponential2}\rm
For any $R>0$, there exist constants $n_{0}=n_{0}(R)\in\mathbb{N}^{+}$ and $0<\delta=\delta(R)<1$ such that for any $n\in\mathbb{Z}$,
\begin{equation}\label{2TV}
\sup_{(x,y)\in\mathcal{C}_{R}}\|P(n-n_{0},x,n,\cdot)-P(n-n_{0},y,n,\cdot)\|_{TV} \leq 2(1-\delta).
\end{equation}
\end{assumption}

\begin{remark}\label{Ref-bound}\rm
\
\begin{enumerate}
\item An alternative way of formulating \eqref{2TV} is that for any $(x,y)\in\mathcal{C}_{R}$ and $n\in\mathbb{Z}$,
\begin{equation*}
\bigg|\int_{\mathcal{X}}\varphi(z)P(n-n_{0},x,n,dz)-\int_{\mathcal{X}}\varphi(z)P(n-n_{0},y,n,dz)\bigg|
\leq 2(1-\delta),
\end{equation*}
uniformly over all measurable functions $\varphi:\mathcal{X}\rightarrow\mathbb{R}$ with $\|\varphi\|_{\infty}\leq 1$.
\item Assumption \ref{Assumption-exponential2} is satisfied if there exist a constant $0<\alpha<1$ and a probability measure $\nu$ such that for any $n\in\mathbb{Z}$ and $x\in\{x:V(x)\leq R\}$,
      \begin{equation*}
      P(n-n_{0},x,n,\cdot) \geq \alpha\nu,
      \end{equation*}
      which is referred to as the minorization condition or the Doeblin's condition localized to level sets of $V$ in \cite{MT} and was used to study invariant measures for homogeneous Markov processes.
\item Assumptions \ref{Assumption-exponential1} and \ref{Assumption-exponential2} ensure that the system is dissipative outside the level set  $\{x:V(x)\leq R_1\}$ for some $R_1>0$ of $V$ and the contraction property of Markov chain on level sets $\{x:V(x)\leq R\}$ for any $R>0$ of $V$, respectively, by which we obtain the exponential ergodicity of Markov chain.
\end{enumerate}
\end{remark}

Denote the weighted supremum norm $\|\cdot\|$ with respect to function $V:\mathcal{X}\rightarrow[0,\infty]$ by
\begin{equation}\label{eqweinorm}
\|\varphi\| := \sup_{x\in\mathcal{X}}\frac{|\varphi(x)|}{1+V(x)}.
\end{equation}
For $\beta>0$, we introduce the definitions of the seminorms $\|\cdot\|_{\beta}$ and $\vvvert\cdot\vvvert_{\beta}$:
\begin{align*}
\|\varphi\|_{\beta} := \sup_{x}\frac{|\varphi(x)|}{1+\beta V(x)},\qquad 
\vvvert\varphi\vvvert_{\beta} := \sup_{x\neq y}\frac{|\varphi(x)-\varphi(y)|}{2+\beta V(x)+\beta V(y)}.
\end{align*}

By the same proof as in \cite{HM2}, we obtain the following Lemma.
\begin{lemma}\label{Lem-norm}
It holds that $\vvvert\varphi\vvvert_{\beta}=\inf_{c\in\mathbb{R}}\|\varphi+c\|_{\beta}$ for any measurable function $\varphi:\mathcal{X}\to\mathbb{R}$.
\end{lemma}

We first derive the contraction property of the transition semigroup $P_{n,n+m},n\in\mathbb{Z},m\in\mathbb{N}$ with the norm $\vvvert\cdot\vvvert_{\beta}$ in discrete time.

\begin{theorem}\label{Th-rate}
Suppose Assumptions \ref{Assumption-exponential1} and \ref{Assumption-exponential2} hold. Then there exist constants $n_{0}\in\mathbb{N}^{+}$, $0<\eta<1$ and $\beta>0$ such that
\begin{equation*}
\bigg\vvvert\int_{\mathcal{X}}\varphi(y)P(n-kn_{0},\cdot,n,dy)\bigg\vvvert_{\beta}
\leq
\eta\bigg\vvvert\int_{\mathcal{X}}\varphi(y)P(n-(k-1)n_{0},\cdot,n,dy)\bigg\vvvert_{\beta},
\end{equation*}
for any $n\in\mathbb{Z}$, $k\in\mathbb{N}^{+}$ and any measurable function $\varphi:\mathcal{X}\to\mathbb{R}$.
\end{theorem}
\begin{proof}
Choose $\gamma^{\ast}\in(\gamma,1)$ and $R$ such that $(1-\gamma)(\gamma^{\ast}-\gamma)R\geq 2C$. By Assumption \ref{Assumption-exponential2} there exist $n_{0}\in\mathbb{N}^{+}$ and $0<\delta<1$ such that \eqref{2TV} holds. We now choose $\beta=\delta/(\gamma R+2 C/(1-\gamma))$.
Fix $\varphi,n,k$ with $\vvvert\int_{\mathcal{X}}\varphi(y)P(n-(k-1)n_{0},\cdot,n,dy)\vvvert_{\beta}<1$. Notice that by Lemma \ref{Lem-norm}, we assume without loss of generality that $\|\int_{\mathcal{X}}\varphi(y)P(n-(k-1)n_{0},\cdot,n,dy)\|_{\beta}<1$. It suffices to prove that for any $x\neq y$,
\begin{equation*}
\bigg|\int_{\mathcal{X}}\varphi(z)P(n-kn_{0},x,n,dz)-\int_{\mathcal{X}}\varphi(z)P(n-kn_{0},y,n,dz)\bigg|
\leq
\eta(2+\beta V(x)+\beta V(y)).
\end{equation*}
For the case of $V(x)+V(y)\geq R$, it holds that
\begin{align*}
&\int_{\mathcal{X}}\varphi(z)P(n-kn_{0},x,n,dz)-\int_{\mathcal{X}}\varphi(z)P(n-kn_{0},y,n,dz)\\
\leq \:&
2+\beta\gamma^{n_{0}} (V(x)+ V(y))+\frac{2\beta C}{1-\gamma}\\
\leq\:&
2+\beta\gamma^{\ast}(V(x)+V(y))+\beta\bigg[\frac{2C}{1-\gamma}-(\gamma^{\ast}-\gamma)(V(x)+V(y))\bigg].
\end{align*}
Reasoning as above we have for some $0<\alpha_{1}<1$,
\begin{equation*}
\int_{\mathcal{X}}\varphi(z)P(n-kn_{0},x,n,dz)-\int_{\mathcal{X}}\varphi(z)P(n-kn_{0},y,n,dz)
\leq
\alpha_{1}[2+\beta(V(x)+V(y))].
\end{equation*}

By the assumption $\|\int_{\mathcal{X}}\varphi(y)P(n-(k-1)n_{0},\cdot,n,dy)\|_{\beta}<1$, we set the function $\int_{\mathcal{X}}\varphi(y)P(n-(k-1)n_{0},x,n,dy)=\varphi_{1}(x)+\varphi_{2}(x)$, where $|\varphi_{1}(x)|\leq 1$ and $|\varphi_{2}(x)|\leq \beta V(x)$. For the case of $V(x)+V(y)< R$, we have
\begin{align*}
&\int_{\mathcal{X}}\varphi(z)P(n-kn_{0},x,n,dz)-\int_{\mathcal{X}}\varphi(z)P(n-kn_{0},y,n,dz)\\
\leq&
\int_{\mathcal{X}}\varphi_{1}(z)(P(n-kn_{0},x,n-(k-1)n_{0},dz)-P(n-kn_{0},y,n-(k-1)n_{0},dz))\\
&
+\bigg|\int_{\mathcal{X}}\varphi_{2}(z)P(n-kn_{0},x,n-(k-1)n_{0},dz)\bigg|\\
&
+\bigg|\int_{\mathcal{X}}\varphi_{2}(z)P(n-kn_{0},y,n-(k-1)n_{0},dz)\bigg|\\
\leq&
\int_{\mathcal{X}}\varphi_{1}(z)(P(n-kn_{0},x,n-(k-1)n_{0},dz)-P(n-kn_{0},y,n-(k-1)n_{0},dz))\\
&
+\gamma\beta V(x)+\gamma\beta V(y)+\frac{2\beta C}{1-\gamma}.
\end{align*}
Therefore, by Remark \ref{Ref-bound} we have
\begin{align*}
&\bigg|\int_{\mathcal{X}}\varphi(z)P(n-kn_{0},x,n,dz)-\int_{\mathcal{X}}\varphi(z)P(n-kn_{0},y,n,dz)\bigg|\\
\leq\:&
2(1-\delta)+\gamma\beta R+\frac{2\beta C}{1-\gamma}\\
\leq\:&
\bigg(1-\frac{\delta}{2}\bigg)(2+\beta V(x)+\beta V(y)).
\end{align*}
We set $\eta=\max\{\alpha_{1},1-\delta/2\}$ and then complete the proof.
\end{proof}

Next we show the existence and uniqueness of the invariant measure family for $P$ in discrete time which satisfies \eqref{eq.discreuni}, on the space of probability measures integrating $V$.

\begin{theorem}\label{thm-uniquet}
Suppose Assumptions \ref{Assumption-exponential1} and \ref{Assumption-exponential2} hold. Then there exists a unique sequence of probability measures $\{\mu_{n}\}_{n\in\mathbb{Z}}$ satisfying $\int_{\mathcal X}V(x)\mu_n(dx)<+\infty$ for any $n\in\mathbb Z$ such that for any $m<n$ and measurable $f$ with $\|f\|\leq 1$,
\begin{equation}\label{eq.discreuni}
\int_{\mathcal X}f(x)P^\ast_{m,n}\mu_m(dx)=\int_{\mathcal X}f(x)\mu_n(dx).
\end{equation}
\end{theorem}
\begin{proof}
By Theorem \ref{Th-rate}, for any $\epsilon>0$, there exists $N\in\mathbb{N}$ such that for any $k\geq N$, $x\neq y$ and measurable $\varphi:\mathcal{X}\to\mathbb{R}$ with $\vvvert\varphi\vvvert_{\beta}\leq 1$,
\begin{equation}\label{eq.beta}
\bigg|\int_{\mathcal{X}}\varphi(z)P(n-kn_{0},x,n,dz)-\int_{\mathcal{X}}\varphi(z)P(n-kn_{0},y,n,dz)\bigg|
\leq
\epsilon(2+\beta V(x)+\beta V(y)).
\end{equation}
Notice that when $f$ is measurable and $\|f\|\leq 1$, namely for any $x\in\mathcal X$, $|f(x)|\leq 1+V(x)$, we have for any $y\neq x$,
\begin{equation*}
\frac{\beta}{2+\beta}|f(x)-f(y)|<2+\beta V(x)+\beta V(y).
\end{equation*}
Therefore, it holds that
\begin{equation}\label{eq.beta1}
\bigg\vvvert \frac{\beta}{2+\beta}f\bigg\vvvert_{\beta}\leq 1.
\end{equation}
Thus we have for any $x\in\{x:V(x)< \infty\}$ and $m>k\geq N$, there exists positive constant $M$ depending on $x$ such that for any measurable function $f$,
\begin{align*}
&\sup_{f:\|f\|\leq1}\bigg|\int_{\mathcal{X}}f(z)P(n-kn_{0},x,n,dz)-\int_{\mathcal{X}}f(z)P(n-mn_{0},x,n,dz)\bigg|\\
\leq\:&
\sup_{f:\|f\|\leq1}\int_{\mathcal{X}}\bigg|\int_{\mathcal{X}}f(z)P(n-kn_{0},x,n,dz)-\int_{\mathcal{X}}f(z)P(n-kn_{0},y,n,dz)\bigg|\\
\:&
P(n-mn_{0},x,n-kn_{0},dy)\\
\leq\:&
\frac{2+\beta}{\beta}\sup_{f:\|f\|\leq 1}\int_{\mathcal{X}}\epsilon(2+\beta V(x)+\beta V(y))P(n-mn_{0},x,n-kn_{0},dy)\\
\leq\:&
\frac{2+\beta}{\beta}\epsilon\bigg(2+\beta V(x)+\beta\gamma^{(m-k)n_{0}} V(x)+ \frac{\beta C}{1-\gamma}\bigg)\\
\leq\:& 
M\epsilon.
\end{align*}
Consequently, the sequence $\{P(n-kn_{0},x,n,\cdot)\}_{k\in\mathbb{N}}$ of probability measures is a Cauchy sequence. Since the space $\mathcal{P}(\mathcal{X})$ is complete with the total variation norm (see the item (iii) of Proposition 8.1 in \cite{BH2022}), using the arguments in the proof of Theorem 3.2 in \cite{HM2}, there exist probability measures $\mu_{n}$ satisfying $\int_{\mathcal X}V(x)\mu_n(dx)<+\infty$ such that $\lim_{k\rightarrow\infty}\sup_{\|f\|\leq 1}\int_{\mathcal X}f(z)(P(n-kn_0,x,n,dz)-\mu_n(dz))=0$ for any $n\in\mathbb{Z}$.

We first prove the existence of sequence of probability measures which satisfies \eqref{eq.discreuni}. Notice that for any $n\in\mathbb Z$, there exists $m\in\mathbb Z$ such that $n\geq mn_0$. Define the sequence of probability measures $\{\nu_n\}_{n\in\mathbb Z}$ by $\nu_n=P^\ast_{mn_0,n}\mu_{mn_0}$.
For measurable $f$ satisfying $\|f\|\leq 1$, namely for any $x\in\mathcal X$, $|f(x)|\leq 1+V(x)$, we have $|\int_{\mathcal X}f(z)P(mn_0,x,n,dz)|\leq(1+C/(1-\gamma))(1+V(x))$ by Assumption \ref{Assumption-exponential1}. Thus
\begin{equation}\label{eq.mn01}
  \bigg\|\int_{\mathcal X}\bigg(1+\frac{C}{1-\gamma}\bigg)^{-1}f(z)P(mn_0,\cdot,n,dz)\bigg\|\leq 1.
\end{equation}
It holds that as $K\to+\infty$, for measurable $f$,
\begin{align}
&\sup_{f:\|f\|\leq 1}\bigg|\int_{\mathcal X}f(z)\nu_n(dz)-\int_{\mathcal X}f(z)\frac{1}{K}\sum_{i=1}^{K}P(mn_0-in_0,x,n,dz)\bigg|\nonumber \\
\leq&\sup_{f:\|f\|\leq 1}\frac{1}{K}\sum_{i=1}^{K}\bigg|\int_{\mathcal X}f(z)\nu_n(dz)-\int_{\mathcal X}f(z)P(mn_0-in_0,x,n,dz)\bigg|\nonumber \\
=&\sup_{f:\|f\|\leq 1}\frac{1}{K}\sum_{i=1}^{K}\bigg|\int_{\mathcal X}f(z)(P^\ast_{mn_0,n}\mu_{mn_0}(dz)-P^\ast_{mn_0,n}P(mn_0-in_0,x,mn_0,dz))\bigg|\nonumber \\
\leq&\sup_{f:\|f\|\leq 1}\frac{1}{K}\sum_{i=1}^{K}\bigg(1+\frac{C}{1-\gamma}\bigg)\bigg|\int_{\mathcal X}f(z)\mu_{mn_0}(dz)-\int_{\mathcal X}f(z)P(mn_0-in_0,x,mn_0,dz)\bigg|\nonumber \\
\to& 0.\label{eq.exis}
\end{align}
We set $\tilde{\nu}_{n}=P_{zn_0,n}^{\ast}\mu_{zn_0}$ for some integer $z$ satisfying $z\neq m$ and $zn_0\leq n$. Using \eqref{eq.exis}, then for any $\epsilon>0$ there exists $N>0$ such that for any $K\geq N$ and measurable $f$,
\begin{align}
&\sup_{f:\|f\|\leq 1}\bigg|\int_{\mathcal X}f(z)\tilde{\nu}_n(dz)-\int_{\mathcal X}f(z)\nu_n(dz)\bigg|\nonumber\\
\leq&\sup_{f:\|f\|\leq 1}\frac{1}{K}\bigg|\sum_{i=1}^{K}\int_{\mathcal X}f(z)P(zn_0-in_0,x,n,dz)-\sum_{i=1}^{K}\int_{\mathcal X}f(z)P(mn_0-in_0,x,n,dz)\bigg|\label{eq.exis1}\\
&+\frac{\epsilon}{2}.\nonumber
\end{align}
Notice that $x$ is fixed.
If $z\leq m$, using Assumption \ref{Assumption-exponential1}, we obtain as $K\to+\infty$,
\begin{align}
 &\sup_{f:\|f\|\leq 1}\frac{1}{K}\bigg|\sum_{i=1}^{K}\int_{\mathcal X}f(z)P(zn_0-in_0,x,n,dz)-\sum_{i=1}^{K}\int_{\mathcal X}f(z)P(mn_0-in_0,x,n,dz)\bigg|\nonumber \\
 \leq&\sup_{f:\|f\|\leq 1}\frac{1}{K}\bigg[\bigg|\sum_{i=1}^{K}\int_{\mathcal X}f(z)\bigg(P(zn_0-in_0,x,n,dz)-P((z+1)n_0-in_0,x,n,dz)\bigg)\bigg|\nonumber \\
 &
 +\ldots+\bigg|\sum_{i=1}^{K}\int_{\mathcal X}f(z)\bigg(P((m-1)n_0-in_0,x,n,dz)-P(mn_0-in_0,x,n,dz)\bigg)\bigg|\bigg]\nonumber \\
 =&\sup_{f:\|f\|\leq 1}\frac{1}{K}\bigg[\bigg|-\int_{\mathcal X}f(z)P(zn_0,x,n,dz)+\int_{\mathcal X}f(z)P(zn_0-Kn_0,x,n,dz)\bigg|+\ldots\nonumber \\
 &+\bigg|-\int_{\mathcal X}f(z)P(mn_0-n_0,x,n,dz)+\int_{\mathcal X}f(z)P((m-1)n_0-Kn_0,x,n,dz)\bigg|\bigg]\nonumber \\
 \leq&\frac{2(1+V(x)+\frac{C}{1-\gamma})}{K}|z-m|\to 0.\label{eq.exis2}
\end{align}
Combining \eqref{eq.exis1} and \eqref{eq.exis2}, $\sup_{f:\|f\|\leq 1}\bigg|\int_{\mathcal X}f(z)\tilde{\nu}_n(dz)-\int_{\mathcal X}f(z)\nu_n(dz)\bigg|\leq \epsilon$ when $K$ is large.
Since $\epsilon$ is arbitrary, $\sup_{f:\|f\|\leq 1}\bigg|\int_{\mathcal X}f(z)\tilde{\nu}_n(dz)-\int_{\mathcal X}f(z)\nu_n(dz)\bigg|=0$.
Thus the definition of $\nu_n$ is independent of the choice of $m$.
It follows from the above argument that $\nu_{n}$ is well defined. 

For each $n_1,n_2\in\mathbb Z$ and $n_1\leq n_2$, there exists $m\in\mathbb Z$ such that $mn_0\leq n_1$. Using \eqref{eq.exis}, for every $\epsilon>0$, there exists $N\in\mathbb N$ such that for $K
\geq N$ and measurable $f$,
\begin{align}
&\sup_{f:\|f\|\leq 1}\bigg|\int_{\mathcal X}f(z)P^\ast_{n_1,n_2}\nu_{n_1}(dz)-\int_{\mathcal X}f(z)\nu_{n_2}(dz)\bigg| \nonumber\\
\leq&\sup_{f:\|f\|\leq 1}\bigg|\int_{\mathcal X}f(z)P^\ast_{n_1,n_2}\nu_{n_1}(dz)-\int_{\mathcal X}f(z)P^\ast_{n_1,n_2}\frac{1}{K}\sum_{i=1}^{K}P(mn_0-in_0,x,n_1,dz)\bigg| \nonumber\\
&+\sup_{f:\|f\|\leq 1}\bigg|\int_{\mathcal X}f(z)P^\ast_{n_1,n_2}\frac{1}{K}\sum_{i=1}^{K}P(mn_0-in_0,x,n_1,dz) \nonumber\\
&-\int_{\mathcal X}f(z)\frac{1}{K}\sum_{i=1}^{K}P(mn_0-in_0,x,n_2,dz)\bigg|
+\frac{\epsilon}{2} \nonumber\\
=&\sup_{f:\|f\|\leq 1}\bigg|\int_{\mathcal X}f(z)P^\ast_{n_1,n_2}\nu_{n_1}(dz)-\int_{\mathcal X}f(z)P^\ast_{n_1,n_2}\frac{1}{K}\sum_{i=1}^{K}P(mn_0-in_0,x,n_1,dz)\bigg|+\frac{\epsilon}{2}.\label{eq.n1}
\end{align}
Using the same arguments for \eqref{eq.mn01}, we have 
\begin{equation}\label{eq.betaC}
\bigg\|\int_{\mathcal X}\big(1+\frac{C}{1-\gamma}\big)^{-1}f(z)P(n_1,\cdot,n_2,dz)\bigg\|\leq 1.
\end{equation}
Combining \eqref{eq.betaC} and \eqref{eq.exis}, as $K\to+\infty$,
\begin{align}
&\sup_{f:\|f\|\leq 1}\bigg|\int_{\mathcal X}f(z)P^\ast_{n_1,n_2}\nu_{n_1}(dz)-\int_{\mathcal X}f(z)P^\ast_{n_1,n_2}\frac{1}{K}\sum_{i=1}^{K}P(mn_0-in_0,x,n_1,dz)\bigg| \nonumber\\
\leq&\sup_{f:\|f\|\leq 1}\big(1+\frac{C}{1-\gamma}\big)\bigg|\int_{\mathcal X}f(z)\nu_{n_1}(dz)-\int_{\mathcal X}f(z)\frac{1}{K}\sum_{i=1}^{K}P(mn_0-in_0,x,n_1,dz)\bigg| \nonumber\\
\to& 0.\label{eq.n0}
\end{align} 
By \eqref{eq.n0} and \eqref{eq.n1}, $\sup_{f:\|f\|\leq 1}\bigg|\int_{\mathcal X}f(z)P^\ast_{n_1,n_2}\nu_{n_1}(dz)-\int_{\mathcal X}f(z)\nu_{n_2}(dz)\bigg|\leq\epsilon$. Since $\epsilon$ is arbitrary, $\{\nu_{n}\}_{n\in\mathbb{Z}}$ satisfies for any $m<n$ and measurable $f$ with $\|f\|\leq 1$,
\begin{equation}\label{eq.exisnu}
\int_{\mathcal X}f(x)P^\ast_{m,n}\nu_m(dx)=\int_{\mathcal X}f(x)\nu_n(dx).
\end{equation}

We now prove the uniqueness of sequence of probability measures satisfying \eqref{eq.exisnu}. For any two sequences of probability measures $\{\mu_{n}\}_{n\in\mathbb{Z}}$ and $\{\nu_{n}\}_{n\in\mathbb{Z}}$ satisfying for any $n\in\mathbb Z$, $\int_{\mathcal X}V(x)\mu_n(dx)<+\infty$ and $\int_{\mathcal X}V(x)\nu_n(dx)<+\infty$, respectively, it suffices to prove that for any measurable $f$ such that $\|f\|\leq 1$ and $n\in\mathbb Z$,
\begin{equation}\label{eq.unique}
  \bigg|\int_{\mathcal X}f(x)\mu_n(dx)- \int_{\mathcal X}f(x)\nu_n(dx)\bigg|=0.
\end{equation}
For any $n\in \mathbb Z$ and measurable $f$ such that $\|f\|\leq 1$,
\begin{align*}
&\bigg|\int_{\mathcal{X}}f(x)\mu_{n}(dx)-\int_{\mathcal{X}}f(x)\nu_{n}(dx)\bigg|\\
=\:&
\bigg|\int_{\mathcal{X}}\int_{\mathcal{X}}f(z)P(n-kn_{0},x,n,dz)(\mu_{n-kn_{0}}(dx)-\nu_{n-kn_{0}}(dx))\bigg|\\
\leq\:&
\int_{\mathcal{X}}\int_{\mathcal{X}}\bigg|\int_{\mathcal{X}}f(z)P(n-kn_{0},x,n,dz)-\int_{\mathcal{X}}f(z)P(n-kn_{0},y,n,dz)\bigg|\\
\:&
\ \ \ \mu_{n-kn_{0}}(dx)\nu_{n-kn_{0}}(dy).
\end{align*}
Using \eqref{eq.beta1} and \eqref{eq.beta}, it holds that
\begin{align*}
&\int_{\mathcal{X}}\int_{\mathcal{X}}\bigg|\int_{\mathcal{X}}f(z)P(n-kn_{0},x,n,dz)-\int_{\mathcal{X}}f(z)P(n-kn_{0},y,n,dz)\bigg|\\
\:&
\ \ \ \mu_{n-kn_{0}}(dx)\nu_{n-kn_{0}}(dy) \\
\leq&\frac{2+\beta}{\beta}\int_{\mathcal{X}}\int_{\mathcal{X}}\epsilon(2+\beta V(x)+\beta V(y))\mu_{n-kn_{0}}(dx)\nu_{n-kn_{0}}(dy) \\
=&\frac{2+\beta}{\beta}\bigg(2\epsilon+\beta\epsilon\int_{\mathcal{X}}V(x)\mu_{n-kn_{0}}(dx)+\beta\epsilon\int_{\mathcal{X}}V(x)\nu_{n-kn_{0}}(dx)\bigg).
\end{align*}
Since $\epsilon$ and $f$ are arbitrary, we obtain $\bigg|\int_{\mathcal{X}}f(x)\mu_{n}(dx)-\int_{\mathcal{X}}f(x)\nu_{n}(dx)\bigg|=0$. The proof of theorem is complete.
\end{proof}

In the following theorem, we present the exponential ergodicity with the norm $\|\cdot\|$ given in \eqref{eqweinorm} of the unique invariant measure family in discrete time, and the exponential ergodicity with the norm $\|\cdot\|_{TV}$ in discrete time.

\begin{theorem}\label{thmTVdiscrete}
Under the conditions in Theorem \ref{thm-uniquet}, for the unique invariant measure family $\{\mu_{n}\}_{n\in\mathbb{Z}}$,
\begin{enumerate}
  \item there exist constants $0<\alpha<1$ and $M>0$ such that
\begin{equation*}
\bigg\|\int_{\mathcal{X}}\varphi(y)P(n-m,\cdot,n,dy)-\int_{\mathcal{X}}\varphi(y)\mu_{n}(dy)\bigg\|
\leq
M\alpha^{m}\bigg\|\varphi-\int_{\mathcal{X}}\varphi(y)\mu_{n}(dy)\bigg\|,
\end{equation*}
for any $n\in\mathbb{Z}$, $m\in\mathbb{N}$ and any bounded measurable function $\varphi:\mathcal{X}\to\mathbb{R}$;
  \item there exist constants $0<\alpha<1$ and $\tilde{M}>0$ such that
\begin{equation*}
\|P(n-m,x,n,\cdot)-\mu_{n}\|_{TV}
\leq
\tilde{M}\alpha^{m}(1+V(x)),
\end{equation*}
for any $n\in\mathbb{Z}$, $m\in\mathbb{N}$ and $x\in\mathcal{X}$.
\end{enumerate}
\end{theorem}
\begin{proof}
(i). Fix $\varphi$ with $\vvvert\int_{\mathcal{X}}\varphi(y)P(n-kn_{0},\cdot,n,dy)\vvvert_{\beta}\leq 1$. Then for any $x\neq y$, we have
\begin{equation*}
\bigg|\int_{\mathcal{X}}\varphi(z)P(n-kn_{0},x,n,dz)-\int_{\mathcal{X}}\varphi(z)P(n-kn_{0},y,n,dz)\bigg| 
\leq 
2+\beta V(x)+\beta V(y).
\end{equation*}
For any $y\in\{y:V(y)< \infty\}$, any $x\in\mathcal{X}$ and some $M_{1}\geq \max\{3+\beta C/(1-\gamma),\beta\}$,
\begin{align*}
&\bigg|\int_{\mathcal{X}}\varphi(z)P(n-kn_{0},x,n,dz)-\lim_{m\rightarrow\infty}\int_{\mathcal{X}}\varphi(z)P(n-mn_{0},y,n,dz)\bigg|\\
\leq\:&
\lim_{m\rightarrow\infty}\int_{\mathcal{X}}(2+\beta V(x)+\beta V(q))P(n-mn_{0},y,n-kn_{0},dq)\\
\leq\:&
M_{1}(1+V(x)).
\end{align*}
Then by Lemma \ref{Lem-norm} and Theorem \ref{Th-rate} we obtain for any bounded measurable function $\varphi:\mathcal{X}\to\mathbb{R}$,
\begin{align*}
\bigg\|\int_{\mathcal{X}}\varphi(z)P(n-kn_{0},\cdot,n,dz)-\int_{\mathcal{X}}\varphi(y)\mu_{n}(dy)\bigg\|
\leq\:& 
M_{1}\bigg\vvvert\int_{\mathcal{X}}\varphi(y)P(n-kn_{0},\cdot,n,dy)\bigg\vvvert_{\beta}\\
\leq\:& 
M_{1}\eta^{k}\inf_{c\in\mathbb{R}}\|\varphi+c\|_{\beta}\\
\leq\:& 
M_{1}\eta^{k}\bigg\|\varphi-\int_{\mathcal{X}}\varphi(y)\mu_{n}(dy)\bigg\|_{\beta}.
\end{align*}
Notice that norms $\|\cdot\|$ and $\|\cdot\|_{\beta}$ are equivalent. Thus there exists a constant $M_{2}$ such that
\begin{equation*}
\bigg\|\int_{\mathcal{X}}\varphi(y)P(n-kn_{0},\cdot,n,dy)-\int_{\mathcal{X}}\varphi(y)\mu_{n}(dy)\bigg\|
\leq 
M_{2}\eta^{k}\bigg\|\varphi-\int_{\mathcal{X}}\varphi(y)\mu_{n}(dy)\bigg\|.
\end{equation*}

It holds that for any $m\in\mathbb{N}$, $m=kn_{0}+q$ with some $k\in\mathbb{N}$ and $q\in\{0,\ldots,n_{0}-1\}$. Fix $\varphi$ with $\|\varphi-\int_{\mathcal{X}}\varphi(y)\mu_{n}(dy)\|\leq 1$. Then we have for $M_{3}\geq 1+C/(1-\gamma)$,
\begin{equation*}
\bigg|\int_{\mathcal{X}}\varphi(y)P(n-q,x,n,dy)-\int_{\mathcal{X}}\varphi(y)\mu_{n}(dy)\bigg|
\leq 
M_{3}(1+V(x)).
\end{equation*}
Therefore, for $\alpha=\eta^{1/n_{0}}$, some $M_{4}\geq M_{2}M_{3}/\eta$ and any bounded measurable function $\varphi:\mathcal{X}\to\mathbb{R}$,
\begin{align*}
\bigg\|\int_{\mathcal{X}}\varphi(y)(P(n-m,\cdot,n,dy)-\mu_{n}(dy))\bigg\|
\leq\:& M_{2}\eta^{k}\bigg\|\int_{\mathcal{X}}\varphi(y)(P(n-q,\cdot,n,dy)-\mu_{n}(dy))\bigg\|\\
\leq\:& 
M_{2}M_{3}\eta^{k}\bigg\|\varphi-\int_{\mathcal{X}}\varphi(y)\mu_{n}(dy)\bigg\|\\
\leq\:
& M_{4}\alpha^{m}\bigg\|\varphi-\int_{\mathcal{X}}\varphi(y)\mu_{n}(dy)\bigg\|.
\end{align*}
The theorem is completely proved.
\medskip

\indent
(ii). By item (i), there exist constants $0<\alpha<1$ and $M>0$ such that for some $\tilde{M}>0$ and any $\varphi:\mathcal{X}\to\mathbb{R}$ with $\|\varphi\|_{\infty}\leq 1$,
\begin{align*}
\sup_{x\in\mathcal{X}}\frac{|\int_{\mathcal{X}}\varphi(z)P(n-m,x,n,dz)-\int_{\mathcal{X}}\varphi(z)\mu_{n}(dz)|}{1+V (x)}
\leq& 
M\alpha^{m}\sup_{x\in\mathcal{X}}\bigg|\varphi(x)-\int_{\mathcal{X}}\varphi(z)\mu_{n}(dz)\bigg|\\
\leq& 
\tilde{M}\alpha^{m}.
\end{align*}
Therefore, it holds that for any $x\in\mathcal{X}$,
\begin{align*}
\|P(n-m,x,n,\cdot)-\mu_{n}\|_{TV}
=&
\sup_{\varphi:\|\varphi\|_{\infty}\leq1}\bigg|\int_{\mathcal{X}}\varphi(z)P(n-m,x,n,dz)-\int_{\mathcal{X}}\varphi(z)\mu_{n}(dz)\bigg|\\
\leq&
\tilde{M}\alpha^{m}(1+V(x)).
\end{align*}
\end{proof}

\begin{remark}\label{Rem-TV}\rm
Under the conditions of Theorem \ref{thmTVdiscrete}, it holds that
\begin{equation*}
\lim_{m\rightarrow\infty}\bigg|\int_{\mathcal{X}}f(y)P(n-m,x,n,dy)-\int_{\mathcal{X}}f(y)\mu_{n}(dy)\bigg| = 0,
\end{equation*}
for any $n\in\mathbb{Z}$, $x\in\{x:V(x)<\infty\}$ and $f:\mathcal{X}\rightarrow\mathbb{R}$ with $\|f\|_{\infty}\leq 1$. Therefore, for any $x\in\{x:V(x)<\infty\}$ and $n\in\mathbb{Z}$,
\begin{equation*}
\lim_{m\rightarrow\infty}P(n-m,x,n,\cdot) = \mu_{n}.
\end{equation*}
\end{remark}

Under the following two assumptions similar to Assumptions \ref{Assumption-exponential1} and \ref{Assumption-exponential2}, respectively, we can establish the exponential ergodicity of invariant measure families for Markov processes.

\begin{assumption}\label{Assumption-exponential1con}\rm
There exists a Borel measurable function $V:\mathcal{X}\to [0,\infty]$ such that,
\begin{enumerate}
\item the set $\{x:V(x)<\infty\}$ is nonempty;
\item for any $t\in\mathbb{R}^{+}$, there exist $0<\gamma=\gamma(t)<1$ and $C>0$ such that for any $x\in \mathcal{X}$ and  $s\in\mathbb{R}$,
      \begin{equation*}
      \int_{\mathcal{X}}V(y)P(s-t,x,s,dy) \leq \gamma V(x)+C.
      \end{equation*}
\end{enumerate}
\end{assumption}

For any $R>0$, define the set $\mathcal{C}_{R}=\{(x,y):V(x)+V(y)\leq R\}$ as before.

\begin{assumption}\label{Assumption-exponential2con}\rm
For any $R>0$, there exist constants $t_{0}=t_0(R)\in\mathbb{R}^{+}$ and $0<\delta=\delta(R)<1$ such that for any $t\in\mathbb{R}$,
\begin{equation}\label{eqassumption318}
\sup_{(x,y)\in\mathcal{C}_{R}}\|P(t-t_{0},x,t,\cdot)-P(t-t_{0},y,t,\cdot)\|_{TV} \leq 2(1-\delta).
\end{equation}
\end{assumption}

\begin{theorem}\label{thm-uniquecontinuous}
Suppose Assumptions \ref{Assumption-exponential1con} and \ref{Assumption-exponential2con} hold. Then there exists a unique invariant measure family $\{\mu_{t}\}_{t\in\mathbb{R}}$ for $P$ satisfying $\int_{\mathcal X}V(x)\mu_t(dx)<+\infty$ for any $t\in\mathbb R$.
\end{theorem}
\begin{proof}
By Assumption \ref{Assumption-exponential2con}, for any $R>0$, there exist constants $t_{0}\in\mathbb{R}^{+}$ and $0<\delta<1$ such that \eqref{eqassumption318} is satisfied.
Notice that $n_{0}< t_{0}\leq n_{0}+1$ for a unique $n_{0}\in\mathbb{N}$.
It holds that for any $t\in\mathbb{R}$ and $R>0$,
\begin{align*}
&\sup_{(x,y)\in\mathcal{C}_{R}}\|P(t-t_{0},x,t-t_{0}+n_{0}+1,\cdot)-P(t-t_{0},y,t-t_{0}+n_{0}+1,\cdot)\|_{TV}\\
\leq\:&
\sup_{(x,y)\in\mathcal{C}_{R}}\sup_{f:\|f\|_{\infty}\leq1}\bigg|\int_{\mathcal{X}}f(z)P(t-t_{0},x,t,dz)-\int_{\mathcal{X}}f(z)P(t-t_{0},y,t,dz)\bigg|\\
\leq\:& 
2(1-\delta).
\end{align*}
Then we have for $t\in\mathbb{R}$,
\begin{equation*}
\sup_{(x,y)\in\mathcal{C}_{R}}\|P(t-n_{0}-1,x,t,\cdot)-P(t-n_{0}-1,y,t,\cdot)\|_{TV} \leq 2(1-\delta).
\end{equation*}
Thus Assumption \ref{Assumption-exponential2} is satisfied. 

By the arguments before \eqref{eq.exisnu} in the proof of Theorem \ref{thm-uniquet}, there exists $\{\nu_n\}_{n\in\mathbb Z}$ satisfying \eqref{eq.discreuni}.
We first prove the existence of invariant measure family $\{\nu_t\}_t$ satisfying for any $t\in\mathbb R$, $\int_{\mathcal X}V(x)\nu_t(dx)<+\infty$. Notice that for any $t\in\mathbb{R}$, there exists $n\in\mathbb{Z}$ such that $t\geq n$. Define $\mu_t=P^\ast_{n,t}\nu_n$. As in the proof of Theorem \ref{thm-uniquet}, we will prove $\mu_t$ is well defined and then $\{\mu_t\}$ is an invariant measure family. We set $\tilde{\mu}_t=P^\ast_{m,t}\nu_m$ for some integer $m$ satisfying $m\neq n$ and $m\leq t$. Let us set $m\leq n$. Notice that there exists $z\in\mathbb Z$ such that $zn_0\leq m$. For measurable $f$,
\begin{align*}
&\sup_{\|f\|\leq 1}\bigg|\int_{\mathcal X}f(z)\tilde{\mu}_t(dz)-\int_{\mathcal X}f(z)\mu_t(dz)\bigg| \\
=&\sup_{\|f\|\leq 1}\bigg|\int_{\mathcal X}f(z)P^\ast_{m,t}\nu_m(dz)-\int_{\mathcal X}f(z)P^\ast_{n,t}\nu_n(dz)\bigg| \\
\leq&\sup_{\|f\|\leq 1}\bigg|\int_{\mathcal X}f(z)P^\ast_{m,t}\nu_m(dz)-\int_{\mathcal X}f(z)P^\ast_{m,t}\frac{1}{K}\sum_{i=1}^{K}P(zn_0-in_0,x,m,dz)\bigg| \\
&+\sup_{\|f\|\leq 1}\bigg|\int_{\mathcal X}f(z)P^\ast_{n,t}\nu_n(dz)-\int_{\mathcal X}f(z)P^\ast_{n,t}\frac{1}{K}\sum_{i=1}^{K}P(zn_0-in_0,x,n,dz)\bigg| \\
&+\sup_{\|f\|\leq 1}\bigg|\int_{\mathcal X}f(z)P^\ast_{m,t}\frac{1}{K}\sum_{i=1}^{K}P(zn_0-in_0,x,m,dz) \\
&-\int_{\mathcal X}f(z)P^\ast_{n,t}\frac{1}{K}\sum_{i=1}^{K}P(zn_0-in_0,x,n,dz)\bigg| \\
=&\sup_{\|f\|\leq 1}\bigg|\int_{\mathcal X}f(z)P^\ast_{m,t}\nu_m(dz)-\int_{\mathcal X}f(z)P^\ast_{m,t}\frac{1}{K}\sum_{i=1}^{K}P(zn_0-in_0,x,m,dz)\bigg| \\
&+\sup_{\|f\|\leq 1}\bigg|\int_{\mathcal X}f(z)P^\ast_{n,t}\nu_n(dz)-\int_{\mathcal X}f(z)P^\ast_{n,t}\frac{1}{K}\sum_{i=1}^{K}P(zn_0-in_0,x,n,dz)\bigg|.
\end{align*} 
Using $\|f\|\leq 1$ and Assumption \ref{Assumption-exponential1con}, we have for $m\in\mathbb Z$ and $z\in\mathcal X$, 
\begin{equation*}
\bigg|\int_{\mathcal X}f(y)P(m,z,t,dy)\bigg|\leq 1+V(x)+C.
\end{equation*}
Thus we obtain
\begin{equation}\label{eq.t1}
\bigg\|\int_{\mathcal X}\frac{1}{1+C}f(y)P(m,\cdot,t,dy)\bigg\|\leq 1.
\end{equation}
Combining \eqref{eq.t1} and \eqref{eq.exis}, as $K\to+\infty$,
\begin{align*}
 &\sup_{\|f\|\leq 1}\bigg|\int_{\mathcal X}f(z)P^\ast_{m,t}\nu_m(dz)-\int_{\mathcal X}f(z)P^\ast_{m,t}\frac{1}{K}\sum_{i=1}^{K}P(zn_0-in_0,x,m,dz)\bigg| \\
&+\sup_{\|f\|\leq 1}\bigg|\int_{\mathcal X}f(z)P^\ast_{n,t}\nu_n(dz)-\int_{\mathcal X}f(z)P^\ast_{n,t}\frac{1}{K}\sum_{i=1}^{K}P(zn_0-in_0,x,n,dz)\bigg| \\
\leq&\sup_{\|f\|\leq 1}(1+C)\bigg|\int_{\mathcal X}f(z)\nu_m(dz)-\int_{\mathcal X}f(z)\frac{1}{K}\sum_{i=1}^{K}P(zn_0-in_0,x,m,dz)\bigg| \\
&+\sup_{\|f\|\leq 1}(1+C)\bigg|\int_{\mathcal X}f(z)\nu_n(dz)-\int_{\mathcal X}f(z)\frac{1}{K}\sum_{i=1}^{K}P(zn_0-in_0,x,n,dz)\bigg| \\
\to& 0.
\end{align*}
Thus $\sup_{\|f\|\leq 1}\bigg|\int_{\mathcal X}f(z)\tilde{\mu}_t(dz)-\int_{\mathcal X}f(z)\mu_t(dz)\bigg|=0$. We have the definition of $\mu_t$ is independent of the choice of $n$ and thus $\mu_t$ is well defined.

For any $s\leq t$, there exists $n\in\mathbb Z$ such that $n\leq s$. We have for measurable $f$, using the definition of $\{\mu_t\}_{t\in\mathbb R}$,
\begin{align*}
&\sup_{\|f\|\leq 1}\bigg|\int_{\mathcal X}f(x)P^\ast_{s,t}\mu_s(dx)-\int_{\mathcal X}f(x)\mu_t(dx)\bigg| \\
\leq&\sup_{\|f\|\leq 1}\bigg|\int_{\mathcal X}f(x)P^\ast_{s,t}\mu_s(dx)-\int_{\mathcal X}f(x)P^\ast_{s,t}P^\ast_{n,s}\nu_n(dx)\bigg|
+\sup_{\|f\|\leq 1}\bigg|\int_{\mathcal X}f(x)P^\ast_{s,t}P^\ast_{n,s}\nu_n(dx) \\
&-\int_{\mathcal X}f(x)P^\ast_{n,t}\nu_n(dx)\bigg|
+\sup_{\|f\|\leq 1}\bigg|\int_{\mathcal X}f(x)P^\ast_{n,t}\nu_n(dx)-\int_{\mathcal X}f(x)\mu_t(dx)\bigg| \\
=&\sup_{\|f\|\leq 1}\bigg|\int_{\mathcal X}f(x)P^\ast_{s,t}\mu_s(dx)-\int_{\mathcal X}f(x)P^\ast_{s,t}P^\ast_{n,s}\nu_n(dx)\bigg|.
\end{align*}
With the similar arguments as for \eqref{eq.t1}, it holds that 
\begin{equation}\label{eq.t2}
\bigg\|\int_{\mathcal X}\frac{1}{1+C}f(y)P(s,\cdot,t,dy)\bigg\|\leq 1.
\end{equation}
Combining \eqref{eq.t2} and the definition of $\{\mu_t\}_{t\in\mathbb R}$, we obtain
\begin{align*}
&\sup_{\|f\|\leq 1}\bigg|\int_{\mathcal X}f(x)P^\ast_{s,t}\mu_s(dx)-\int_{\mathcal X}f(x)P^\ast_{s,t}P^\ast_{n,s}\nu_n(dx)\bigg| \\
\leq&\sup_{\|f\|\leq 1}(1+C)\bigg|\int_{\mathcal X}f(x)\mu_s(dx)-\int_{\mathcal X}f(x)P^\ast_{n,s}\nu_n(dx)\bigg| \\
=&0.
\end{align*}
Thus $\sup_{\|f\|\leq 1}\bigg|\int_{\mathcal X}f(x)P^\ast_{s,t}\mu_s(dx)-\int_{\mathcal X}f(x)\mu_t(dx)\bigg|=0$. We obtain $\{\mu_t\}_{t\in\mathbb R}$ is an invariant measure family.

We now prove the uniqueness of invariant measure family. For any two invariant measure families $\{\mu_{t}\}_{t\in\mathbb{R}}$ and $\{\nu_{t}\}_{t\in\mathbb{R}}$ satisfying for any $t\in\mathbb R$, $\int_{\mathcal X}V(x)\mu_t(dx)<+\infty$ and $\int_{\mathcal X}V(x)\nu_t(dx)<+\infty$, respectively, it suffices to prove that for any measurable $f$ such that $\|f\|\leq 1$,
\begin{equation}\label{eq.unique0}
  \bigg|\int_{\mathcal X}f(x)\mu_t(dx)- \int_{\mathcal X}f(x)\nu_t(dx)\bigg|=0.
\end{equation}
For any $t\in\mathbb R$, there exists a unique $n\in\mathbb Z$ such that $n\leq t<n+1$, and for measurable $f$,
\begin{equation*}
\sup_{f:\|f\|\leq 1}\bigg|\int_{\mathcal X}f(x)\mu_t(dx)-\int_{\mathcal X}f(x)\int_{\mathcal X}P(n,y,t,dx)\mu_n(dy)\bigg|=0,
\end{equation*}
\begin{equation*}
\sup_{f:\|f\|\leq 1}\bigg|\int_{\mathcal X}f(x)\nu_t(dx)-\int_{\mathcal X}f(x)\int_{\mathcal X}P(n,y,t,dx)\nu_n(dy)\bigg|=0.
\end{equation*}
Therefore, 
\begin{align}
&\sup_{f:\|f\|\leq 1}\bigg|\int_{\mathcal X}f(x)\mu_t(dx)-\int_{\mathcal X}f(x)\nu_t(dx)\bigg|\nonumber \\
\leq&\sup_{f:\|f\|\leq 1}\bigg|\int_{\mathcal X}f(x)\mu_t(dx)-\int_{\mathcal X}f(x)\int_{\mathcal X}P(n,y,t,dx)\mu_n(dy)\bigg|\nonumber \\
&+\sup_{f:\|f\|\leq 1}\bigg|\int_{\mathcal X}f(x)\nu_t(dx)-\int_{\mathcal X}f(x)\int_{\mathcal X}P(n,y,t,dx)\nu_n(dy)\bigg|\nonumber \\
&+\sup_{f:\|f\|\leq 1}\bigg|\int_{\mathcal X}f(x)\int_{\mathcal X}P(n,y,t,dx)\mu_n(dy)-\int_{\mathcal X}f(x)P(n,y,t,dx)\nu_n(dy)\bigg|\nonumber \\
=&\sup_{f:\|f\|\leq 1}\bigg|\int_{\mathcal X}f(x)\int_{\mathcal X}P(n,y,t,dx)\mu_n(dy)-\int_{\mathcal X}f(x)P(n,y,t,dx)\nu_n(dy)\bigg|.\label{eq.uniquet}
\end{align}
Since $\|f\|\leq 1$, using the Assumption \eqref{Assumption-exponential1con} and the fact $\gamma<1$, we have for any $x\in\mathcal X$,
\begin{align*}
\frac{1}{1+C}\bigg|\int_{\mathcal X}f(y)P(n,x,t,dy)\bigg|
\leq&\frac{1}{1+C}\int_{\mathcal X}|f(y)|P(n,x,t,dy) \\
\leq&\frac{1}{1+C}\int_{\mathcal X}(1+V(y))P(n,x,t,dy) \\
\leq&\frac{1}{1+C}(1+V(x)+C)\\
\leq&1+V(x).
\end{align*}
Thus 
\begin{equation}\label{eq.leq1}
\bigg\|\frac{1}{1+C}\int_{\mathcal X}f(y)P(n,\cdot,t,dy)\bigg\|\leq 1.
\end{equation}
Using \eqref{eq.uniquet}, \eqref{eq.leq1} and \eqref{eq.unique}, we obtain \eqref{eq.unique0}. The invariant measure family is unique.
\end{proof}

We present the exponential ergodicity with the norm $\|\cdot\|$ given in \eqref{eqweinorm} of the unique invariant measure family, and the exponential ergodicity with the norm $\|\cdot\|_{TV}$.

\begin{theorem}\label{thmTV}
Under the conditions in Theorem \ref{thm-uniquecontinuous}, for the unique invariant measure family $\{\nu_{t}\}_{t\in\mathbb{R}}$,
\begin{enumerate}
  \item there exist constants $0<\alpha<1$ and $M>0$ such that
\begin{equation*}
\bigg\|\int_{\mathcal{X}}\varphi(y)P(\tau-t,\cdot,\tau,dy)-\int_{\mathcal{X}}\varphi(y)\mu_{\tau}(dy)\bigg\|
\leq
M\alpha^{t}\bigg\|\varphi-\int_{\mathcal{X}}\varphi(y)\mu_{\tau}(dy)\bigg\|,
\end{equation*}
for any $\tau\in\mathbb{R}$, $t\geq 0$ and any bounded measurable function $\varphi:\mathcal{X}\to\mathbb{R}$;
  \item there exist constants $0<\alpha<1$ and $\tilde{M}>0$ such that
\begin{equation}\label{eqTV}
\|P(\tau-t,x,\tau,\cdot)-\mu_{\tau}\|_{TV}
\leq
\tilde{M}\alpha^{t}(1+V(x)),
\end{equation}
for any $\tau\in\mathbb{R}$, $t\geq 0$ and $x\in\mathcal{X}$.
\end{enumerate}
\end{theorem}
\begin{proof}
(i). For any $\tau\in\mathbb{R}$, we have $\tau\in[n,n+1)$ with a unique $n\in\mathbb{Z}$. Then for any $m\in\mathbb{N}$, it holds that
\begin{equation*}
n-m-1 \leq \tau-m-1 < n-m \leq n \leq \tau < n+1.
\end{equation*}
By item (i) in Theorem \ref{thmTVdiscrete}, there exist constants $0<\alpha<1$ and $M>0$ such that for any bounded function $\varphi:\mathcal{X}\to\mathbb{R}$,
\begin{align*}
&\bigg\|\int_{\mathcal{X}}\varphi(y)P(n-m,\cdot,\tau,dy)-\int_{\mathcal{X}}\varphi(y)\mu_{\tau}(dy)\bigg\|\\
=\:&
\bigg\|\int_{\mathcal{X}}\varphi(y)\int_{\mathcal{X}}P(n-m,\cdot,n,dz)P(n,z,\tau,dy)-\int_{\mathcal{X}}\varphi(y)\int_{\mathcal{X}}P(n,z,\tau,dy)\mu_{n}(dz)\bigg\|\\
\leq \:& 
M\alpha^{m}\bigg\|\int_{\mathcal{X}}\varphi(y)(P(n,\cdot,\tau,dy)-\mu_{\tau}(dy))\bigg\|.
\end{align*}
Fix $\varphi$ with $\|\varphi-\int_{\mathcal{X}}\varphi(z)\mu_{\tau}(dz)\|\leq 1$. It holds that
\begin{equation*}
\bigg|\varphi(x)-\int_{\mathcal{X}}\varphi(z)\mu_{\tau}(dz)\bigg| \leq 1+V(x).
\end{equation*}
By Assumption \ref{Assumption-exponential1con}, we have
\begin{equation*}
\bigg|\int_{\mathcal{X}}\varphi(z)P(n,x,\tau,dz)-\int_{\mathcal{X}}\varphi(z)\mu_{\tau}(dz)\bigg| \leq M_{1}(1+V(x)),
\end{equation*}
for any $x\in\mathcal{X}$ and $M_{1}\geq1+C/(1-\gamma)$. Thus for any bounded measurable function $\varphi:\mathcal{X}\to\mathbb{R}$,
\begin{equation*}
\bigg\|\int_{\mathcal{X}}\varphi(z)P(n,\cdot,\tau,dz)-\int_{\mathcal{X}}\varphi(z)\mu_{\tau}(dz)\bigg\|
\leq 
M_{1}\bigg\|\varphi-\int_{\mathcal{X}}\varphi(z)\mu_{\tau}(dz)\bigg\|.
\end{equation*}
Using the same arguments, it holds that
\begin{equation*}
\bigg\|\int_{\mathcal{X}}\varphi(z)(P(\tau-m-1,\cdot,\tau,dz)-\mu_{\tau}(dz))\bigg\|
\leq 
M_{1}\bigg\|\int_{\mathcal{X}}\varphi(z)(P(n-m,\cdot,\tau,dz)-\mu_{\tau}(dz))\bigg\|.
\end{equation*}
Therefore, for any $M_{2}\geq M_{1}^{2}M/\alpha$,
\begin{equation*}
\bigg\|\int_{\mathcal{X}}\varphi(z)P(\tau-m-1,\cdot,\tau,dz)-\int_{\mathcal{X}}\varphi(z)\mu_{\tau}(dz)\bigg\|
\leq 
M_{2}\alpha^{m+1}\bigg\|\varphi-\int_{\mathcal{X}}\varphi(z)\mu_{\tau}(dz)\bigg\|.
\end{equation*}
Consequently, for any $m\in\mathbb{N}$,
\begin{equation*}
\bigg\|\int_{\mathcal{X}}\varphi(z)(P(\tau-m,\cdot,\tau,dz)-\mu_{\tau}(dz))\bigg\|
\leq 
M_{2}\alpha^{m}\bigg\|\varphi-\int_{\mathcal{X}}\varphi(z)\mu_{\tau}(dz)\bigg\|.
\end{equation*}

For any $t\geq 0$, we have $t\in[q,q+1)$ with a unique $q\in\mathbb{N}$ and $t=q+s$ with some $s\in[0,1)$. Therefore, for any bounded function $\varphi:\mathcal{X}\to\mathbb{R}$,
\begin{align*}
&\bigg\|\int_{\mathcal{X}}\varphi(z)P(\tau-t,\cdot,\tau,dz)-\int_{\mathcal{X}}\varphi(z)\mu_{\tau}(dz)\bigg\|\\
=\:&
\bigg\|\int_{\mathcal{X}}\varphi(z)(\int_{\mathcal{X}}P(\tau-t,\cdot,\tau-s,dy)P(\tau-s,y,\tau,dz)-\int_{\mathcal{X}}P(\tau-s,y,\tau,dz)\mu_{\tau-s}(dy))\bigg\|\\
\leq\:& M_{2}\alpha^{q}\bigg\|\int_{\mathcal{X}}\varphi(z)P(\tau-s,\cdot,\tau,dz)-\int_{\mathcal{X}}\varphi(z)\int_{\mathcal{X}}P(\tau-s,y,\tau,dz)\mu_{\tau-s}(dy)\bigg\|\\
\leq\:& 
M_{1}M_{2}\alpha^{q}\bigg\|\varphi-\int_{\mathcal{X}}\varphi(z)\mu_{\tau}(dz)\bigg\|\\
<\:&
\frac{M_{1}M_{2}}{\alpha}\alpha^{t}\bigg\|\varphi-\int_{\mathcal{X}}\varphi(z)\mu_{\tau}(dz)\bigg\|.
\end{align*}
\medskip

\indent
(ii). By item (i), there exist a unique invariant measure family $\{\mu_{t}\}_{t\in\mathbb{R}}$ for $P$, constants $0<\alpha<1$ and $M>0$ such that for some $\tilde{M}>0$ and any $\varphi:\mathcal{X}\to\mathbb{R}$ with $\|\varphi\|_{\infty}\leq 1$,
\begin{align*}
\sup_{x\in\mathcal{X}}\frac{|\int_{\mathcal{X}}\varphi(z)P(\tau-t,x,\tau,dz)-\int_{\mathcal{X}}\varphi(z)\mu_{\tau}(dz)|}{1+V (x)}
\leq\:& 
M\alpha^{t}\sup_{x\in\mathcal{X}}\bigg|\varphi(x)-\int_{\mathcal{X}}\varphi(z)\mu_{\tau}(dz)\bigg|\nonumber\\
\leq\:& 
\tilde{M}\alpha^{t}.
\end{align*}
Therefore, it holds that for any $x\in\mathcal{X}$,
\begin{align*}
\|P(\tau-t,x,\tau,\cdot)-\mu_{\tau}\|_{TV}
=\:&
\sup_{\varphi:\|\varphi\|_{\infty}\leq1}\bigg|\int_{\mathcal{X}}\varphi(z)P(\tau-t,x,\tau,dz)-\int_{\mathcal{X}}\varphi(z)\mu_{\tau}(dz)\bigg|\nonumber\\
\leq\:&
\tilde{M}\alpha^{t}(1+V(x)).
\end{align*}
\end{proof}

\begin{remark}\label{Rem-n}\rm
Under the conditions of Theorem \ref{thmTV}, it holds that
\begin{equation*}
\lim_{t\rightarrow\infty}\bigg|\int_{\mathcal{X}}f(y)P(\tau-t,x,\tau,dy)-\int_{\mathcal{X}}f(y)\mu_{\tau}(dy)\bigg| = 0,
\end{equation*}
for any $\tau\in\mathbb{R}$, $x\in\{x:V(x)<\infty\}$ and $f:\mathcal{X}\rightarrow\mathbb{R}$ with $\|f\|_{\infty}\leq 1$. Therefore, for any $x\in\{x:V(x)<\infty\}$ and $\tau\in\mathbb{R}$,
\begin{equation*}
\lim_{t\rightarrow\infty}P(\tau-t,x,\tau,\cdot) = \mu_{\tau}.
\end{equation*}
\end{remark}

\section{Applications}\label{sec5}

In this section, we illustrate our theoretical results for Markov chains on a countable state space, diffusion processes and storage processes.

\subsection{Markov chains on a countable state space}\label{Sec-Markovchains}

It is worthwhile illustrating our results for Markov chains on a countable state space which include many Markovian models. Specifically, we show the ergodic properties of Markov chains.

Let $\mathcal{X}=\{a_{1},\ldots,a_{l},\ldots\}$ be the countable state space. Consider Markov chain $\{X_{n};n\geq k\}$ for any $k\in\mathbb{Z}$ on $\mathcal{X}$ with the transition probabilities $P(i,x,j,\cdot)$ and transition operators $P^{\ast}_{i,j}$ defined by
\begin{equation*}
P^{\ast}_{i,j}\nu(\cdot) = \sum_{x\in\mathcal{X}}P(i,x,j,\cdot)\nu(x),
\end{equation*}
for any integers $i,j$ such that $j\geq i\geq k$, $x\in\mathcal{X}$ and probability measure $\nu$ on $\mathcal{X}$, where $\nu(x):=\nu(\{x\})$.
A family of probability measures $\{\mu_{n}\}_{n\in\mathbb{Z}}$ on $\mathcal{X}$ is said to be an invariant measure family for $\{X_{n};n\geq k\}$ if for each $k\in\mathbb{Z}$ and $i\geq k$,
\begin{equation*}
P^{\ast}_{k,i}\mu_{k} = \mu_{i}.
\end{equation*}

\begin{definition} \rm\label{def.chain}
We say Markov chain $\{X_{n};n\geq k\}$ for any $k\in\mathbb{Z}$ is \emph{uniformly elliptic} on set $A$, if there exist a Borel probability measure $m$, Borel measurable functions $p_{n}:\mathcal{X}\times\mathcal{X}\rightarrow[0,\infty)$ and a constant $0<\epsilon_{0}<1$ such that for any $n\in \mathbb{Z}$,
\begin{enumerate}
\item $P(n-1,z,n,y)=p_{n-1}(z,y)m(y)$ for any $z,y\in\mathcal{X}$;
\item $\inf_{z\in\mathcal{X}}\sum_{y\in\mathcal{X}}p_{n}(x,y)p_{n+1}(y,z)m(y)\geq\epsilon_{0}$ for any $x\in A$,
\end{enumerate}
where $P(n-1,z,n,y):=P(n-1,z,n,\{y\})$, $m(y):=m(\{y\})$ for any $z,y\in\mathcal{X}$ and $n\in \mathbb{Z}$.
\end{definition}

\begin{remark}\rm
We give an example which satisfies items in Definition \ref{def.chain}.  
Let $\mathcal{X}=\{a_{1},\ldots,a_{N}\}$ be the finite state space, $m^\ast$ be the counting measure and probability measure $m=m^\ast/N$. Denote $NP(n-1,x,n,y)$ by $p_{n-1}(x,y)$ for any $x,y\in\mathcal X$ and $n\in\mathbb Z$. Assume for any $x,y\in\mathcal X$, $P(n,x,n+1,y):=P(n,x,n+1,\{y\})$ is periodic with respect to $n$ with period $T_{x,y}\in\mathbb N$, and $P(i,x,i+1,y)>0$ for $i=0,\ldots,T_{x,y}-1$. 

With these notations, for any $x,y\in\mathcal X$ there exists $T_{x,y}^\ast\in\mathbb N$ such that, $P(n,x,n+2,y)$ is periodic with respect to $n$ with period $T_{x,y}^\ast$. Indeed, we can choose $T_{x,y}^\ast$ equal to
\begin{equation*}
\bigg[[T_{x,a_1},T_{a_1,y}],\ldots,[T_{x,a_N},T_{a_N,y}]\bigg],
\end{equation*}
where $[a,b]$ is the least common multiple of $a,b$.
\medskip

\noindent
(i). It holds that for any $z,y\in\mathcal X$ and $n\in\mathbb Z$,
\begin{align*}
P(n-1,z,n,y)&=P(n-1,z,n,y)N\frac{m^\ast(y)}{N} \\
&=p_{n-1}(z,y)m(y),
\end{align*}
where $m^\ast(y):=m^\ast(\{y\})$ and $m(y):=m(\{y\})$. 
\medskip

\noindent
(ii). For any $n\in\mathbb Z$ and $x,y\in\mathcal X$, there exist $k\in\mathbb Z$ and $p\in\{0,1,\ldots,T^\ast_{x,y}-1\}$ such that $n=kT^\ast_{x,y}+p$. By assumptions, $P(n+mT^\ast_{x,y},x,n+2+mT^\ast_{x,y},y)=P(n,x,n+2,y)$ for $m\in\mathbb Z$. Therefore,
\begin{equation*}
\min_{n\in\mathbb Z}P(n,x,n+2,y)=\min_{n\in\{0,\ldots,T_{x,y}^\ast-1\}}P(n,x,n+2,y).
\end{equation*}
It holds that for $A\subset\mathcal X$ and $x\in A$,
\begin{align*}
&\inf_{z\in\mathcal X}\sum_{y\in\mathcal X}p_n(x,y)p_{n+1}(y,z)m(y) \\
=&\inf_{z\in\mathcal X}\sum_{y\in\mathcal X}N^2P(n,x,n+1,y)P(n+1,y,n+2,z)\frac{1}{N} \\
=&\inf_{z\in\mathcal X}NP(n,x,n+2,z) \\
\geq&\min_{z\in\mathcal X}N\min_{n\in\{0,\ldots,T_{x,z}^\ast-1\}}P(n,x,n+2,z) \\
>&0.
\end{align*}
Take $\epsilon_0:=\min_{z\in\mathcal X}N\min_{n\in\{0,\ldots,T_{x,z}^\ast-1\}}P(n,x,n+2,z)$. We obtain (ii) in Definition \ref{def.chain}.
\end{remark}

We will establish the existence, uniqueness and exponential ergodicity with the norm $\|\cdot\|_{TV}$ of the invariant measure family for $\{X_{n};n\geq k\}$ in the following theorem.

\begin{theorem}\label{thmchain}
Suppose there exist a Borel measurable function $V:\mathcal{X}\to[0,\infty)$, constants $0<\gamma<1$ and $C>0$ such that for any $x\in \mathcal{X}$ and $n\in\mathbb{Z}$,
\begin{align}\nonumber
\sum_{y\in\mathcal{X}}V(y)P(n-1,x,n,y) \leq \gamma V(x)+C,
\end{align}
and for any $R>0$, Markov chain $\{X_{n};n\geq k\}$ for any $k\in\mathbb{Z}$ is uniformly elliptic on set $\{x:V(x)\leq R\}$. Then there exist a unique invariant measure family $\{\mu_{n}\}_{n\in\mathbb{Z}}$ for $\{X_{n};n\geq k\}$ satisfying $\int_{\mathcal X}V(x)\mu_n(dx)<+\infty$ for any $n\in\mathbb Z$, constants $0<\alpha<1$ and $\tilde{M}>0$ such that
\begin{equation*}
\|P(n-m,x,n,\cdot)-\mu_{n}\|_{TV}
\leq
\tilde{M}\alpha^{m}(1+V(x)),
\end{equation*}
for any $n\in\mathbb{Z}$, $m\in\mathbb{N}$ and $x\in\mathcal{X}$.
\end{theorem}
\begin{proof}
For any $A\subset\mathcal{X}$, $n\in\mathbb{Z}$, $R>0$, and $x\in\{x:V(x)\leq R\}$,
\begin{align*}
P(n-2,x,n,A)=
\:&
\sum_{y\in\mathcal{X}}P(n-2,x,n-1,y)\sum_{z\in A}p_{n-1}(y,z)m(z) \\
=\:&
\sum_{z\in A}\sum_{y\in\mathcal{X}}p_{n-2}(x,y)p_{n-1}(y,z)m(y)m(z) \\
\geq\:&
\epsilon_{0}m(A).
\end{align*}
Therefore, for any $a_{i}\in\{x:V(x)\leq R\}$ there exists a measure $\mu_{a_{i}}$ such that
\begin{equation*}
P(n-2,a_{i},n,A) = \epsilon_{0}m(A)+(1-\epsilon_{0})\mu_{a_{i}}(A).
\end{equation*}
Then we have for any $a_{i},a_{j}\in\{x:V(x)\leq R\}$,
\begin{equation*}
\|P(n-2,a_{i},n,\cdot)-P(n-2,a_{j},n,\cdot)\|_{TV}=(1-\epsilon_{0})\|\mu_{a_{i}}-\mu_{a_{j}}\|_{TV}\leq 2(1-\epsilon_{0}).
\end{equation*}
Thus Assumption \ref{Assumption-exponential2} is satisfied. The proof is concluded by Theorems \ref{thm-uniquet} and \ref{thmTVdiscrete}.
\end{proof}

\begin{remark}\rm
Dolgopyat and Sarig \cite{DO} studied local limit theorems for inhomogeneous Markov chains with uniformly elliptic property on the whole state space $\mathcal{X}$ and bounded densities $p_{n}$ for $n\geq1$. Given a nonnegative function $V$, the uniformly elliptic assumption in \cite{DO} is relaxed to the level sets of $V$ when establishing the ergodicity of Markov chains for $n\geq k$ in Theorem \ref{thmchain}. In the example provided in \cite{DY}, the densities $p_{n}$ for $n\geq1$, satisfying the uniformly elliptic condition on $\mathcal{X}$, vanish on a large set. Extending $n$ to be in $\mathbb{Z}$ and taking $V(x)=x^{2}$, this specific Markov chain has a unique invariant measure family with convergence in total variation distance by Theorem \ref{thmchain}.
\end{remark}

\subsection{Diffusion processes on $\mathbb{R}^{n}$}\label{sec4.2}

The focus of this subsection is on the diffusion process $\{X(t);t\geq s\}$ for any $s\in\mathbb{R}$, the special case of Markov processes which exists globally on $\mathbb{R}^{n}$ and has almost surely continuous sample paths. For any $t\geq\tau\geq s$, denote by $P(\tau,x,t,\cdot)$ the transition probabilities of $\{X(t);t\geq s\}$, by $P^{\ast}_{\tau,t}$ the transition operators of $\{X(t);t\geq s\}$ given in \eqref{Def-transition} and by $P_{\tau,t}$ the transition semigroups of $\{X(t);t\geq s\}$ given in \eqref{Def-semigroup}.

Suppose that $X(\cdot)$ satisfies the stochastic differential equation
\begin{equation}\label{SDE}
d X(t) =  f(t, X(t)) d t + g(t, X(t))d B(t),
\end{equation}
where both $f:\mathbb{R}\times\mathbb{R}^{n}\to\mathbb{R}^{n}$ and $g:\mathbb{R}\times\mathbb{R}^{n}\to\mathcal{M}^{n\times n}$ are continuous, $f$ is the drift vector, the diffusion matrix $g$ is symmetric and nonnegative definite, and $B$ is an $n$-dimensional Wiener process. Let $\mathcal{M}^{n\times n}$ denote the set of all $n\times n$-matrices, equipped with the norm of $\mathbb{R}^{nn}$.
For any $(t,x)\in\mathbb{R}\times\mathbb{R}^{n}$, denote $|f(t,x)|_{2}=(\sum_{i=1}^{n}(f_{i}(t,x))^{2})^{1/2}$ and $\|g(t,x)\|_{2}=(\sum_{i,j=1}^{n}(g_{ij}(t,x))^{2})^{1/2}$.

In the following, we show the existence, uniqueness and exponential ergodicity of the invariant measure family given in Section \ref{sec2} for \eqref{SDE}. Recall the definition and properties of uniformly almost periodic functions given in Section \ref{sec2}. 

We say function $F:\mathbb R\times\mathbb R^n\to\mathbb R$ belongs to the domain $\mathcal D(\mathcal L)$ of the generator $\mathcal L$ of $X(t)$ given in \eqref{SDE} if the following limit exists (see Definition 5.3 in \cite{Sarkka2019Solin}):
\begin{equation*}
(\mathcal L F)(X(t),t)=\lim_{s\downarrow 0}\frac{E[F(X(t+s),t+s)|X(t)]-F(X(t),t)}{s}.
\end{equation*}
If $F(t,x)$ is twice continuously differentiable with respect to $x$ and continuously differentiable with respect to $t$ then $F\in\mathcal D(\mathcal L)$ (see arguments before $(3.35)$ in \cite{Kh}).
The generator $\mathcal L$ associated with \eqref{SDE} is then given as (see Definition 5.3 in \cite{Sarkka2019Solin}) 
\begin{equation*}
\mathcal{L}
=\frac{\partial}{\partial t}+\sum_{i=1}^{n}\sum_{j=1}^{n}a_{ij}(t,x)\frac{\partial^{2}}{\partial x_{i}\partial x_{j}}+\sum_{i=1}^{n}f_{i}(t,x)\frac{\partial}{\partial x_{i}},
\end{equation*}
where $a_{ij}=(1/2g\cdot g^{\top})_{ij}$.
Let $(\mathcal{X}_{1},d_{1})$ and $(\mathcal{X}_{2},d_{2})$ be Polish spaces. Denote by $C^{0,k}(\mathbb{R}\times\mathcal{X}_{1},\mathcal{X}_{2})$ the space of all continuous functions $f:\mathbb{R}\times\mathcal{X}_{1}\rightarrow\mathcal{X}_{2}$ which are $k$ times continuously differentiable on $\mathcal{X}_{1}$ for $k\in\mathbb{N}^{+}$.

\begin{definition} \rm\label{Def-Lyaeqcon}
A twice continuously differentiable function $V:\mathcal{X}\to[0,\infty)$ is called a Lyapunov function for \eqref{SDE} if $V$ satisfies the following conditions:
\begin{enumerate}
\item for any $R\geq 0$, $\{x:V(x)\leq R\}$ is compact;
\item there exists $c>0$ such that for any $x\in \mathcal{X}$ and $t\in\mathbb{R}$,
      \begin{equation*}
      (\mathcal{L}V)(t,x) \leq -cV(x).
      \end{equation*}
\end{enumerate}
\end{definition}
It holds that the function $V$ in Definition \ref{Def-Lyaeqcon} is in $\mathcal D(\mathcal L)$.

\begin{theorem}\label{Th-nondegenerate}
Let $f$ and $g$ be of classes $C^{0,1}(\mathbb{R}\times\mathbb{R}^{n},\mathbb{R}^{n})$ and $C^{0,2}(\mathbb{R}\times\mathbb{R}^{n},\mathcal{M}^{n\times n})$, respectively. Suppose that the following conditions hold:
\begin{enumerate}
\item there exists a Lyapunov function $V(x)$ for \eqref{SDE};
\item $f(t,x)$ and $g(t,x)$ are almost periodic in $t$ uniformly for $x\in\mathbb{R}^{n}$;
\item there exists $\lambda>0$ such that for any $\xi\in\mathbb{R}^{n}$, $\sum_{i,j=1}^{n}(gg^{\top})_{ij}\xi_{i}\xi_{j}\geq\lambda|\xi|^{2}$.
\end{enumerate}
Then there exists a unique invariant measure family $\{\mu_{t}\}_{t\in\mathbb{R}}$ for \eqref{SDE} satisfying $\int_{\mathcal X}V(x)\mu_t(dx)<+\infty$ for any $t\in\mathbb R$. Furthermore, there exist constants $0<\alpha<1$ and $\tilde{M}>0$ such that
\begin{equation*}
\|P(\tau-t,x,\tau,\cdot)-\mu_{\tau}\|_{TV}
\leq
\tilde{M}\alpha^{t}(1+V(x)),
\end{equation*}
for any $\tau\in\mathbb{R}$, $t\geq 0$ and $x\in\mathbb{R}^{n}$.
\end{theorem}
\begin{proof}
Denote by $B(0,n)$ the ball of radius $n>0$ centered at $0$. It holds that for every $n\in\mathbb{N}$, there exists $K_{n}<\infty$ such that for all $s\in \mathbb{R}$ and $x,y\in B(0,n)$,
\begin{equation*}
\|g(s,x)\|_{2}+|f(s,x)|_{2} \leq K_{n}(1+|x|_{2}),
\end{equation*}
and
\begin{equation*}
\|g(s,x)-g(s,y)\|_{2}+|f(s,x)-f(s,y)|_{2} \leq K_{n}|x-y|_{2}.
\end{equation*}
Then by Theorem 3.4 in \cite{Kh}, there exists a unique solution $X(\cdot)$ of \eqref{SDE}.
For any $t\in\mathbb{R}$ and $x\in\mathbb{R}^{n}$, take $u(t,x)=V(x)e^{ct}$. It holds that
\begin{align*}
\mathcal{L}u(t,x)
=\:&
cV(x)e^{ct}+\sum_{i=1}^{n}\sum_{j=1}^{n}a_{ij}(t,x)\frac{\partial^{2}V(x)}{\partial x_{i}\partial x_{j}}e^{ct}+\sum_{i=1}^{n}f_{i}(t,x)\frac{\partial V(x)}{\partial x_{i}}e^{ct}\\
=\:&
e^{ct}(cV(x)+(\mathcal{L}V)(t,x))\\
\leq\:& 
0.
\end{align*}
For any $R>0$, denote by $\tau_{R}$ the time at which the sample function of $X$ first leaves the compact set $\{x:V(x)\leq R\}$ and let $\tau_{R}(t)=\min\{\tau_{R},t\}$. By Dynkin's formula, for any $s\leq\tau_{R}(t)$ and $X(s)=x\in\{x:V(x)\leq R\}$ we have
\begin{equation*}
E[u(X(\tau_{R}(t)),\tau_{R}(t))] 
= u(x,s)+E\int_{s}^{\tau_{R}(t)}\mathcal{L}u(X(\tau),\tau)d\tau
\leq u(x,s),
\end{equation*}
where $E$ is the expectation with respect to the law of $X$ given $X(s)=x$.
Letting $R\rightarrow\infty$, we have $\tau_{R}(t)\rightarrow t$. It holds that
\begin{align*}
e^{ct}\int_{\mathbb{R}^{n}}V(y)P(s,x,t,dy)=\:&E[u(X(t),t)]\\
\leq\:&
\liminf_{R\rightarrow\infty}E[u(X(\tau_{R}(t)),\tau_{R}(t))]\\
\leq\:& 
V(x)e^{cs}.
\end{align*}
Thus we have for any $x\in\mathbb{R}^{n}$, $s\in\mathbb{R}$ and $t>0$,
\begin{equation*}
\int_{\mathbb{R}^{n}}V(y)P(s-t,x,s,dy) \leq e^{-ct}V(x).
\end{equation*}
Then Assumption \ref{Assumption-exponential1con} is satisfied with $\gamma(t)=e^{-ct}$ and any $C>0$.

It is well known that $P(s,x,t,A)>0$ for any $s\leq t$, $x\in\mathbb{R}^{n}$ and nonempty open set $A\subset\mathbb{R}^{n}$.
By Remark 3.9 in \cite{Kh}, for any $s<t$, $x\in\mathbb{R}^{n}$ and $\Gamma\subset\mathbb{R}^{n}$ we have
\begin{equation*}
P(s,x,t,\Gamma) = \int_{\Gamma}p(s,x,t,y)dy,
\end{equation*}
where $p(s,x,t,y)$ is the unique positive and continuous density.

For every pair of sequences $\{\alpha'_{n}\},\{\beta'_{n}\}\subset\mathbb{R}$, there exist common subsequences $\{\alpha_{n}\}$ and $\{\beta_{n}\}$ such that for any compact set $K\subset\mathcal\mathbb{R}^{n}$, $T_{\alpha+\beta}f$, $T_{\alpha}T_{\beta}f$, $T_{\alpha+\beta}g$ and $T_{\alpha}T_{\beta}g$ exist, and $T_{\alpha+\beta}f=T_{\alpha}T_{\beta}f$, $T_{\alpha+\beta}g=T_{\alpha}T_{\beta}g$ uniformly on $\mathbb{R}\times K$.
By Theorem 3.4 in \cite{Kh}, equations with coefficients $T_{\alpha}f$, $T_{\alpha}g$ and $T_{\alpha}T_{\beta}f$, $T_{\alpha}T_{\beta}g$, respectively, have unique solutions. For any $s\in\mathbb{R}$, $t\in\mathbb{R}^{+}$ and $x\in\mathbb{R}^{n}$, denote transition probability functions of the unique solutions by $P_{X}(s-t,x,s,\cdot)$ and $P_{Y}(s-t,x,s,\cdot)$, respectively. Then by Theorem 3 in \cite{GS}, we have for any compact set $K\subset\mathbb{R}^{n}$, $y\in\mathbb{R}^{n}$ and $t\in\mathbb{R}^{+}$,
\begin{equation*}
\lim_{n\rightarrow\infty}P(s-t+\alpha_{n},x,s+\alpha_{n},dy)-P_{Y}(s-t,x,s,dy) = 0,
\end{equation*}
and
\begin{equation*}
\lim_{n\rightarrow\infty}P_{Y}(s-t+\beta_{n},x,s+\beta_{n},dy)-P_{Z}(s-t,x,s,dy) = 0,
\end{equation*}
uniformly on $\mathbb{R}\times K$.
Using the same arguments as above we have the equation with coefficients $T_{\alpha+\beta}f$ and $T_{\alpha+\beta}g$ has a unique solution. Denote the transition probability function of the unique solution by $P_{Z_{1}}(s-t,x,s,\cdot)$. Thus we have
\begin{equation*}
\lim_{n\rightarrow\infty}P(s-t+\alpha_{n}+\beta_{n},x,s+\alpha_{n}+\beta_{n},dy)-P_{Z_{1}}(s-t,x,s,dy) = 0,
\end{equation*}
uniformly on $\mathbb{R}\times K$.
Since the solution of the equation with coefficients $T_{\alpha+\beta}f$ and $T_{\alpha+\beta}g$ is unique, it holds that
\begin{equation*}
\lim_{m\rightarrow\infty}\lim_{n\rightarrow\infty}P(s-t+\alpha_{n}+\beta_{m},x,s+\alpha_{n}+\beta_{m},dy)
= 
\lim_{n\rightarrow\infty}P(s-t+\alpha_{n}+\beta_{n},x,s+\alpha_{n}+\beta_{n},dy),
\end{equation*}
uniformly on $\mathbb{R}\times K$.
Thus for any $y\in\mathbb{R}^{n}$ and $t\in\mathbb{R}^{+}$, $P(s-t,x,s,dy)$ is almost periodic in $s$ uniformly for $x\in\mathbb{R}^{n}$ by Proposition \ref{prop-almostperiodic}. Therefore, $p(s-t,x,s,y)$ is almost periodic in $s$ uniformly for $x\in\mathbb{R}^{n}$.

For any $R>0$, take
\begin{equation}\label{eqR1}
R_{1} \geq \max\{R, 2C/(1-\gamma(R^{2}))\}.
\end{equation}
Fixing $y^{\ast}\in\mathbb{R}^{n}$ and take $R_{2}>R_{1}$ such that $O(R_{2})\supset\{x:V(x)\leq R_{1}\}$, where $O(R_{2})$ is the open cube of edge length $R_{2}$ centered at $y^{\ast}$. Notice that for any $y\in\mathbb{R}^{n}$, $p(t-9R_{2}^{2},x,t,y)$ are almost periodic in $t$ uniformly for $x\in\mathbb{R}^{n}$. For each $\epsilon>0$ there exist an $l\in\mathbb{R}^{+}$ and $s\in[-t+9R_{2}^{2},-t+9R_{2}^{2}+l]$ satisfying $|p(t+s-9R_{2}^{2},x,t+s,y)-p(t-9R_{2}^{2},x,t,y)|<\epsilon$ for any $t\in\mathbb{R}$, $x\in\{x:V(x)\leq R\}$ and $y\in O(R_{2})$. Thus we have
\begin{equation*}
p(t-9R_{2}^{2},x,t,y)
> p(t+s-9R_{2}^{2},x,t+s,y)-\epsilon
\geq \inf_{9R_{2}^{2}\leq q\leq 9R_{2}^{2}+ l}p(q-9R_{2}^{2},x,q,y)-\epsilon.
\end{equation*}
Therefore, for the Lebesgue measure $\lambda$, any $t\in\mathbb{R}$, $x\in\{x:V(x)\leq R\}$ and $A\subset\mathbb{R}^{n}$,
\begin{align*}
&P(t-9R_{2}^{2},x,t,A) \\
\geq\:&
P(t-9R_{2}^{2},x,t,A\cap O(R_{2})) \\
\geq\:&
\inf_{x\in\{x:V(x)\leq R\}}\inf_{y\in O(R_{2})}p(t-9R_{2}^{2},x,t,y)\lambda(A\cap O(R_{2})) \\
\geq\:&
(\inf_{x\in\{x:V(x)\leq R\}}\inf_{y\in O(R_{2})}\inf_{9R_{2}^{2}\leq q\leq 9R_{2}^{2}+ l}p(q-9R_{2}^{2},x,q,y)-\epsilon)\lambda(A\cap O(R_{2})).
\end{align*}

Notice that $l$ depends on the choice of $\epsilon$. Let us show that for some small $\epsilon>0$, $\inf_{x\in\{x:V(x)\leq R\}}\inf_{y\in O(R_{2})}\inf_{9R_{2}^{2}\leq q\leq 9R_{2}^{2}+ l}p(q-9R_{2}^{2},x,q,y)-\epsilon>0$.
Assume
$0<\inf_{x\in\{x:V(x)\leq R\}}\inf_{y\in O(R_{2})}\inf_{9R_{2}^{2}\leq q\leq 9R_{2}^{2}+ l}p(q-9R_{2}^{2},x,q,y)\leq\epsilon$. It holds that for some $9R_{2}^{2}\leq q^{\ast}\leq 9R_{2}^{2}+ l$,
\begin{equation*}
\inf_{x\in\{x:V(x)\leq R\}}\inf_{y\in O(R_{2})}p(q^{\ast}-9R_{2}^{2},x,q^{\ast},y) \leq \epsilon.
\end{equation*}
Denote $Q(R_{2})=O(R_{2})\times(q^{\ast}-R_{2}^{2},q^{\ast})$. We have
\begin{align*}
&\inf_{x\in\{x:V(x)\leq R\}}\inf_{(y,t)\in Q(R_{2})}p(q^{\ast}-9R_{2}^{2},x,t,y) \\
\leq\:&
\inf_{x\in\{x:V(x)\leq R\}}\inf_{y\in O(R_{2})}p(q^{\ast}-9R_{2}^{2},x,q^{\ast},y) \\
\leq\:&
\epsilon.
\end{align*}
According to the continuity of $f$, $g$ and $\partial g/\partial x_{i}$ and the parabolic Harnack inequality stated as Theorem 6.2.7 (or the analogous result in Theorem 8.1.1) in \cite{BKRS}, there exist positive constants $C_{0}$ and $q_{1}\in[q^{\ast}-8R_{2}^{2},q^{\ast}-7R_{2}^{2}]$ such that
\begin{equation*}
\sup_{y\in O(R_{2})}\inf_{x\in\{x:V(x)\leq R\}}p(q^{\ast}-9R_{2}^{2},x,q_{1},y) \leq C_{0}\epsilon.
\end{equation*}
We can choose $\epsilon$ small enough so that $C_{0}\epsilon< (1-\gamma(R^{2}))/2$. It holds that
\begin{align*}
&\inf_{x\in\{x:V(x)\leq R\}}P(q^{\ast}-9R_{2}^{2},x,q_{1},\{x:V(x)\leq R_{1}\}) \\
\leq\:&
\inf_{x\in\{x:V(x)\leq R\}}P(q^{\ast}-9R_{2}^{2},x,q_{1},O(R_{2})) \\
<\:&
\frac{1-\gamma(R^{2})}{2}.
\end{align*}
Notice that we have by Chebyshev's inequality and \eqref{eqR1},
\begin{equation*}
\inf_{x\in\{x:V(x)\leq R\}}P(q^{\ast}-9R_{2}^{2},x,q_{1},\{x:V(x)\leq R_{1}\})
\geq 
1-\frac{\gamma(R_{2}^{2}) R+C}{R_{1}}
\geq 
\frac{1-\gamma(R^{2})}{2},
\end{equation*}
which is a contradiction. Reasoning as above we have
\begin{equation*}
a := \inf_{x\in\{x:V(x)\leq R\}}\inf_{y\in O(R_{2})}p(q^{\ast}-9R_{2}^{2},x,q^{\ast},y) > \epsilon.
\end{equation*}
It holds that
\begin{equation*}
P(t-9R_{2}^{2},x,t,A) \geq \delta\nu(A),
\end{equation*}
with $\delta=(a-\epsilon)\min\{\lambda(O(R_{2})),1\}$ and $\nu(A)=\lambda(A\cap O(R_{2}))/\lambda(O(R_{2}))$. Therefore, for any $x\in\{x:V(x)\leq R\}$ there exists a measure $\mu_{x}$ such that for any $t\in\mathbb{R}$ and $A\subset\mathbb{R}^{n}$,
\begin{equation*}
P(t-9R_{2}^{2},x,t,A) = \delta\nu(A)+(1-\delta)\mu_{x}(A).
\end{equation*}
Then we have for any $x,y\in\{x:V(x)\leq R\}$,
\begin{equation*}
\|P(t-9R_{2}^{2},x,t,\cdot)-P(t-9R_{2}^{2},y,t,\cdot)\|_{TV}
=
(1-\delta)\|\mu_{x}-\mu_{y}\|_{TV}
\leq 
2(1-\delta).
\end{equation*}
Thus Assumption \ref{Assumption-exponential2con} is satisfied. The proof is concluded by item (ii) in Theorem \ref{thmTV}.
\end{proof}

\begin{coro}\label{corodiffusion}
Under the conditions in Theorem \ref{Th-nondegenerate}, there exists a unique invariant measure family $\{\mu_{t}\}_{t\in\mathbb{R}}$ such that $\mu_{t}$ is almost periodic in $t$.
\end{coro}
\begin{proof}
By Theorem \ref{Th-nondegenerate}, there exists a unique invariant measure family $\{\mu_{t}\}_{t\in\mathbb{R}}$.
Notice that by Theorem \ref{Th-nondegenerate}, for any $\tau\in\mathbb{R}^{+}$ and $x\in\mathbb{R}^{n}$, $P(t-\tau,x,t,\cdot)$ is almost periodic in $t$ and by Remark \ref{Rem-n}, there exists a sequence $\{\tau_{n}\}\subset\mathbb{R}^{+}$ such that for any $x\in\mathbb{R}^{n}$ and $t\in\mathbb{R}$,
\begin{equation*}
\lim_{n\rightarrow\infty}P(t-\tau_{n},x,t,\cdot) = \mu_{t}.
\end{equation*}
It holds that the above convergence is uniformly with respect to $t$ by \eqref{eqTV}. The proof is concluded by Property $3$ given in \cite{LZ}.
\end{proof}

\subsection{Storage processes with the almost periodic release rule}\label{sec4.3}

Storage processes with the Brownian input have received much attention. Harrison and Taylor \cite{JA} investigated the optimal control of Brownian storage systems. Harrison and Williams \cite{JR} considered a Brownian model of a multiclass service station. In this subsection we study ergodic properties of storage processes with the Brownian input.
We investigate the storage model in which the storage level $X_{t}$ at time $t$ satisfies almost surely the equation
\begin{equation}\label{eq-Brownian}
X_{t} =  x+B_{t}-\int_{\tau}^{t}r(s,X_{s})ds, \quad X_{\tau} = x,\quad t \geq \tau,
\end{equation}
where $\{B_{t};t\in\mathbb{R}\}$ is a 1-dimensional Brownian motion and the release function $r(t,x)$ is nonnegative and almost periodic in $t$ uniformly for $x\in[0,\infty)$.

Let $r$ be of the class $C^{0,1}(\mathbb{R}\times[0,\infty),[0,\infty))$. Assume that there exists $\lambda>0$ such that for any $(t,x)\in\mathbb{R}\times[0,\infty)$, $r(t,x)\geq\lambda$. If there exists a Lyapunov function $V$ given in Definition \ref{Def-Lyaeqcon} for \eqref{eq-Brownian}, then there exists a unique invariant measure family $\{\mu_{t}\}_{t\in\mathbb{R}}$ for \eqref{eq-Brownian} with convergence in total variation distance by Theorem \ref{Th-nondegenerate}. Furthermore, $\mu_{t}$ is almost periodic in $t$ by Corollary \ref{corodiffusion}.

A specific release function $r(t,x)$ is considered in the following example.
\begin{example}\rm
Consider the storage model
\begin{equation}\label{eq-storage}
X_{t} =  x_{\tau}+B_{t}-\int_{\tau}^{t}(\sin\pi s+\cos\sqrt{2}\pi s+X_{s}+3)ds,\quad X_{\tau} = x_{\tau},\quad t \geq \tau,
\end{equation}
where $r(s,x)=\sin\pi s+\cos\sqrt{2}\pi s+x+3$ is almost periodic in $s$ uniformly for $x\in[0,\infty)$.

Take the Lyapunov function $V(x)=x,x\in[0,\infty)$, then it suffices to verify that $V$ satisfies the second condition in Definition \ref{Def-Lyaeqcon}. It holds that $(\mathcal{L}V)(t,x)=-(\sin\pi t+\cos\sqrt{2}\pi t+x+3)\leq-x$. Then there exist a unique invariant measure family $\{\mu_{t}\}_{t\in\mathbb{R}}$ for \eqref{eq-storage} satisfying $\int_{\mathcal X}V(x)\mu_t(dx)<+\infty$ for any $t\in\mathbb R$, constants $0<\alpha<1$ and $\tilde{M}>0$ such that
\begin{equation*}
\|P(\tau-t,x,\tau,\cdot)-\mu_{\tau}\|_{TV}
\leq
\tilde{M}\alpha^{t}(1+V(x)),
\end{equation*}
for any $\tau\in\mathbb{R}$, $t\geq 0$ and $x\in[0,\infty)$ by Theorem \ref{Th-nondegenerate}. Furthermore, $\mu_{t}$ is almost periodic in $t$ by Corollary \ref{corodiffusion}.

\begin{remark}\rm
We will study the existence, uniqueness and exponential ergodicity of invariant measure family for equation \eqref{SDE} with recurrent coefficients in our future work. It can be seen in \cite{CL2} and \cite{CL} for the notations and properties of recurrent functions.
\end{remark}

\end{example}

\section*{Acknowledgements}

The authors would like to thank the anonymous referees for their careful reading and valuable comments which greatly improved this paper. 
This work is supported by National Key R\&D Program of China (No. 2023YFA1009200), NSFC Grant 11925102, and Liaoning Revitalization Talents Program (Grant XLYC2202042).

\end{document}